\documentclass[10pt]{amsart}

\title{On the packing dimension of unions and extensions of $k$-planes}

\author{Jacob B. Fiedler}
\address{Department of Mathematics, University of Wisconsin-Madison, Wisconsin 53715}

\email{jbfiedler2@wisc.edu}
\thanks{The author was supported in part by NSF DMS-2037851 and NSF DMS-2246906.}

\subjclass[2020]{28A78, 28A80, 68Q30}	

\usepackage{amsmath}
\usepackage{amssymb}
\usepackage{mdframed}
\usepackage{amsthm}
\usepackage{mathtools}
\usepackage{enumitem}
\usepackage{placeins}
\usepackage{url}
\usepackage{hyperref}

\newtheorem{thm}{Theorem}

\newtheorem{lem}[thm]{Lemma}
\newtheorem{prop}[thm]{Proposition}
\newtheorem{cor}[thm]{Corollary}

\theoremstyle{remark}
\newtheorem*{remark}{Remark}

\newtheorem*{T1}{Theorem~\ref{thm:planeUnion}}
\newtheorem*{T2}{Theorem~\ref{thm:purePackingBound}}
\newtheorem*{T3}
{Theorem~\ref{thm:basicExtension}}
\newtheorem*{P4}
{Proposition~\ref{prop:trivialRadialSlice}}
\newtheorem*{T5}
{Theorem~\ref{thm:optimalExtension}}
\newtheorem*{T6}
{Theorem~\ref{thm:unionHyperplanes}}

\DeclareMathOperator{\Dim}{Dim}

\newcommand{\R}{\mathbb{R}}

\newcommand{\N}{\mathbb{N}}
\newcommand{\Q}{\mathbb{Q}}

\newcommand{\ve}{\varepsilon}

\newenvironment{proofof}[1]{\begin{trivlist}
		\item[\hskip \labelsep \textit{Proof of #1.}]}{\end{trivlist}}

\begin{document}
\begin{abstract}
    We study the packing dimension of unions of subsets of $k$-planes in $\mathbb{R}^n$ using tools from algorithmic information theory, obtaining an analog of a result of H\'era and a mild generalization of a recent result of Fraser. Along the way, we introduce a notion of effective dimension on the Grassmannian and affine Grassmannian, and we establish several useful algorithmic and geometric tools in this setting. Additionally, we consider how the packing dimension of the union of certain subsets of $k$-planes changes when the subsets are extended to the entire $k$-plane. Finally, we improve the above bounds for unions and extensions in the special case that $k=n-1$.
\end{abstract}
\maketitle

\section{Introduction}

Many problems in geometric measure theory are particular instances of the following general (and somewhat vague) question: 

\noindent \textbf{Question:} How large are unions of subsets of $k$-planes in $\mathbb{R}^n$? More specifically, assuming 
\begin{equation*}   F=\bigcup_{P\in\mathcal{P}}P^\prime
\end{equation*}
where $\mathcal{P}$ is a set of $k$-planes and each $P^\prime\subseteq P$, how can one bound the size of $F$ in terms of the size of $\mathcal{P}$ and some lower bound on the size of each $\mathcal{P}^\prime$?

The conjectures arising from specifying the notions of size involved include:

\begin{itemize}
    \item If $k=1$, $\mathcal{P}$ is large enough that its elements include a line in every direction, and each $\mathcal{P}^\prime$ contains a unit line segment, then the Kakeya conjecture (recently resolved in $\mathbb{R}^3$ by Wang and Zahl \cite{wang2025volumeestimatesunionsconvex}) is that $F$ has Hausdorff dimension $n$.
    \item If $k\geq 2$, $\mathcal{P}$ is large enough that its elements include a translate of every $k$-plane passing through the origin, and each $P^\prime$ contains a unit disc, the $(n, k)$-Besicovitch conjecture is that $F$ has positive measure. 
    \item In $\mathbb{R}^2$, if $\mathcal{P}$ is a set of affine lines with Hausdorff dimension at least $t$, and each $P^\prime$ has Hausdorff dimension at least $s$, the Furstenberg set conjecture (resolved by Orponen-Shmerkin and Ren-Wang \cite{orponen2023projectionsfurstenbergsetsabc, ren2023furstenberg}) asserted that 
    \begin{equation*}
        \dim_H(F)\geq \min\{s+t, s+1, \dfrac{3s+t}{2}\}.
    \end{equation*}
\end{itemize}
In this paper, we consider bounds of this general type that involve \emph{packing} dimension. 

\subsection{Unions of subsets of $k$-planes}

Throughout, $\mathcal{G}(n, k)$ will denote the Grassmannian and $\mathcal{A}(n, k)$ will denote the affine Grassmannian. One of our main theorems is the following Furstenberg-type bound: 

\begin{thm}\label{thm:planeUnion}
    Suppose $\mathcal{P}\subseteq\mathcal{A}(n, k)$ has packing dimension $t>0$. Let $F=\bigcup_{P\in\mathcal{P}}P^\prime$, where each $P^\prime\subseteq P$ has Hausdorff dimension at least $s>0$. Then,
    \begin{equation*}
        \dim_P(F)\geq s + \frac{t - (k - \lceil s\rceil)(n-k)}{\lceil s\rceil +1}.
    \end{equation*}
    In particular, the union of a packing dimension $t$ set of $k$-planes has packing dimension at least $\frac{t}{k+1} + k$.
\end{thm}

This is a packing dimension analog of a result of H\'era \cite{Hera2019}, who proved the same bound but with both instances of packing dimension replaced by Hausdorff dimension. Note that for all $E\subseteq\mathbb{R}^n$, $\dim_P(E)\geq \dim_H(E)$. Hence, the novelty in our result is that we bound the packing dimension of the union in terms of the packing dimension of the set of $k$-planes; both our hypothesis and our conclusion are weaker than in H\'era's result. 

H\'era constructed several examples illustrating the sharpness of her bound for certain values of $n, k, s,$ and $t$. In particular, the example given by Proposition 1.13 in \cite{Hera2019} shows that for unions of $k$-planes, the bound $\frac{t}{k+1} + k$ is sharp when $t=i(k+1)$ for $i\in\mathbb{N}$ and is in general close to sharp. We note that this proposition makes no use of any property of Hausdorff dimension that packing dimension lacks, hence its packing dimension analog implies that Theorem \ref{thm:planeUnion} is sharp at the same values. 

For other values of $t$, H\'era conjectured that her example was optimal, generalizing a conjecture of D. Oberlin \cite{doberlin2014kplanes}. This conjecture was resolved in the case of lines by Zahl \cite{zahl2022} and for general $k$ by Gan \cite{gan2023hausdorffdimensionunionskplanes}, but we are unaware of any results on the packing dimension analog of H\'era's conjecture (aside from Theorem \ref{thm:planeUnion}). 

One feature of Theorem \ref{thm:planeUnion} is that it is not a ``pure'' packing dimension bound, in that we consider the \emph{Hausdorff} dimension of subsets of each $k$-plane. In fact, the version of Theorem \ref{thm:planeUnion} that stipulates each $P^\prime$ has packing dimension at least $s$ is false; Fraser gave examples in the case $k=1$ \cite{fraser2024variantsfurstenbergsetproblem}. The reason for this discrepancy is that, intuitively, packing dimension measures the size of a set at the scales at which the set appears the largest, whereas Hausdorff dimension does so for the scales at which it appears the smallest. Since the scales at which the set of $k$-planes is large may differ from the scales at which the subsets of $k$-planes are large, the two packing dimensions may fail to add. Fraser considered this problem for lines in $\mathbb{R}^n$ \cite{fraser2024variantsfurstenbergsetproblem}. In particular, he proved that given $E\subseteq\mathbb{R}^n$, if there is a collection of lines $\mathcal{L}$ of packing dimension $t$ such that each $L\in\mathcal{L}$ intersects $E$ in a set of packing dimension $s$, then
\begin{equation*}
    \dim_P (E) \geq \max\{s, t/2\}.
\end{equation*}
Furthermore, Fraser gave an example showing that this bound is sharp. Our next theorem is a generalization of this lower bound from lines to $k$-planes.

\begin{thm}\label{thm:purePackingBound}
        Suppose $\mathcal{P}\subseteq\mathcal{A}(n, k)$ has packing dimension $t>0$. Let $F=\bigcup_{P\in\mathcal{P}}P^\prime$, where each $P^\prime\subseteq P$ has packing dimension at least $s>0$. Then,
    \begin{equation*}
        \dim_P(F)\geq \max\{s, \frac{t- (k-\lceil s\rceil)(n-k)}{\lceil s\rceil +1}\}.
    \end{equation*}
\end{thm}
For the purpose of completeness, we note that we actually prove slightly stronger results than Theorem \ref{thm:planeUnion} and Theorem \ref{thm:purePackingBound} as stated above. If we are able to make additional non-concentration assumptions on the $P^\prime\subset P$, then the lower bounds potentially improve. These conditions and improvements are discussed in the remark at the end of Section 4. 

Now, we briefly describe some of the history of this problem (which has mostly concerned Hausdorff dimension). The aforementioned $(n, k)$-Besicovitch conjecture is an early example this kind of problem; Bourgain proved that the conjecture holds for $k$ above a certain threshold \cite{bourgain1991nkbesicovitch}, and R. Oberlin achieved the same bound with a different proof that also gave more precise information about the associated maximal operator \cite{roberlin2007nkBesicovitch}.  D. Oberlin gave bounds on the Hausdorff dimension of unions of $k$-planes and hyperplanes \cite{doberlin2014kplanes, doberlin2007hyperplanes}. Additional results pertaining to hyperplanes include \cite{falconer2016strong} and \cite{dabrowski2022}. Wolff was the first to consider bounding the size of unions of fractal subsets of lines, introducing the Furstenberg set problem \cite{Wolff1999RecentWC}, and Molter and Rela made explicit the notion of an $(s, t)$-Furstenberg set \cite{molter2012furstenberg}.

For $t\in[0, 1]$, H\'era, Keleti, and Mathe showed \cite{heraKeleti2019hausdorff} that if $\mathcal{P}\subseteq\mathcal{A}(n, k)$ has Hausdorff dimension $t$, and $F=\bigcup_{P\in\mathcal{P}}P^\prime$ where each $P^\prime\subseteq P$ has Hausdorff dimension at least $s$, then,
    \begin{equation*}
        \dim_H(F)\geq 2s - k + \min\{t, 1\}.
    \end{equation*}
This bound and the aforementioned work of H\'era seem to be the first instances of the ``full'' fractal version of Furstenberg set problem in higher dimensions, that is, the $(s, t, k)$-Furstenberg set problem. Most recently, Bright and Dhar bounded the size of ``spread'' Furstenberg sets, where the $k$-planes are required to lie in different directions \cite{bright2025spreadfurstenbergsets}. We especially recommend \cite{bright2025spreadfurstenbergsets} to interested readers; in addition to providing more background on this kind of problem, its approach also draws on several of the different results we have referenced above. 

As previously indicated, Fraser proved sharp bounds on the packing dimension of variants of Furstenberg sets \cite{fraser2024variantsfurstenbergsetproblem} (In private communication, Bushling also noted a short alternative proof of Fraser's lower bound using effective dimension). Aside from Fraser's work, packing dimension does not seem to have featured much in the higher-dimensional version of the Furstenberg set problem. However, we note that packing dimension has occasionally played a more explicit role in the related exceptional set problem for orthogonal projections. For details, see \cite{Bushling24Projections, orponen2015exceptional, FalconerHowroyd1996}.

\subsection{Extensions of subsets of $k$-planes}

Our next results also concern the size of unions of subsets of $k$-planes, but from a somewhat different perspective. Here, we are interested in how the size of the union can \emph{change} when certain subsets of each $k$-plane are extended to the entire $k$-plane. In particular, we establish the following result.

\begin{thm}\label{thm:basicExtension}
Let $E\subseteq\mathbb{R}^n$ and let $F$ denote the union of $E$ with every $k$-plane that intersects $E$ in a set of positive measure. Then either $E=F$, or  

\begin{equation*}
    \dim_P(F)\leq 2 \dim_P(E) - k.
\end{equation*}
\end{thm}

This is a generalization of Proposition 1.5 in \cite{bushling2025extension} in two senses. First, we generalize from lines to $k$-planes. Second, instead of extending only subsets of $k$-planes with nonempty interior, we extend every subset with positive measure. In other words, we are (potentially) adding more to the original set because in Theorem \ref{thm:basicExtension}, less is required to add a given $k$-plane. Hence, the upper bound says (a small amount) more. Furthermore, assuming $k=1$, we are able to further strengthen this result by proving the same bounds for extensions not just of every positive measure subset of a line, but of every Hausdorff dimension $1$ subset of a line. More precisely,

\begin{prop}\label{prop:trivialRadialSlice}
Let $E\subseteq\mathbb{R}^n$ and let $F$ denote the union of $E$ with every line that intersects $E$ in a set of Hausdorff dimension 1. Then either $E=F$, or   
    \begin{equation*}
    \dim_P(F)\leq 2 \dim_P(E) - 1.
\end{equation*}
\end{prop}

The extension problem was originally studied in the context of extending line segments to full lines. Keleti showed that the Hausdorff dimension of a collection of line segments in the plane does not increase when replaced with the corresponding full lines, conjectured that this is true in $\mathbb{R}^n$ as well, and established a connection to the Kakeya conjecture \cite{keleti2016lines}. In particular, a resolution of the ``line segment extension conjecture'' in $\mathbb{R}^n$ would prove that Besicovitch sets have Hausdorff dimension at least $n-1$. Later, Keleti and Mathe proved that if the Kakeya conjecture holds in $\mathbb{R}^n$, then the line segment extension conjecture also holds \cite{keleti2022equivalences}. Hence, Wang and Zahl's recent proof of the Kakeya conjecture in $\mathbb{R}^3$ also established the line segment extension conjecture in $\mathbb{R}^3$ \cite{wang2025volumeestimatesunionsconvex}. Falconer and Mattila proved a stronger version of the conjecture for hyperplanes, in the sense that they showed extending every positive measure subset of hyperplanes does not increase the Hausdorff dimension of the original set \cite{falconer2016strong}. Moreover,  \cite{heraKeleti2019hausdorff} implies the same even when extending every full Hausdorff dimension set of hyperplanes. 

Keleti's original paper left open whether the \emph{packing} dimension of a set can increase under line segment extension; Bushling and the author proved that this is impossible even if every Hausdorff dimension $1$ subset of a line is extended \cite{bushling2025extension}.

\subsection{Results for hyperplanes}

In the final section of this paper, we focus on extensions and unions in the special case of hyperplanes. First, we generalize the main result of \cite{bushling2025extension} from lines in $\mathbb{R}^2$ to hyperplanes in $\mathbb{R}^n$. In particular, we show that

\begin{thm}\label{thm:optimalExtension}
    Let $E\subseteq\mathbb{R}^n$ be given, and let $F$ denote the union of $E$ with every $(n-1)$-plane that intersects $E$ in a set of Hausdorff dimension $n-1$. Then, 
    \begin{equation*}
        \dim_P(F)=\dim_P(E).
    \end{equation*}
\end{thm}

\noindent Note that this is an improvement over the $n-1$ case of Theorem \ref{thm:basicExtension}; it is also a packing dimension analog of the aforementioned extension results in \cite{falconer2016strong} and \cite{heraKeleti2019hausdorff}. Using much of the same machinery, we are also able to improve upon Theorem \ref{thm:planeUnion} in the case of full-dimension subsets of hyperplanes. Namely, we prove

\begin{thm}\label{thm:unionHyperplanes}
 Suppose $\mathcal{P}\subseteq\mathcal{A}(n, n-1)$ has packing dimension $t>0$. Let $F=\bigcup_{P\in\mathcal{P}}P^\prime$, where each $P^\prime\subseteq P$ has Hausdorff dimension $n-1$. Then,
    \begin{equation*}
        \dim_P(F)\geq n-1 + \frac{nt}{(n-1)t +n}.
    \end{equation*}
\end{thm}
\noindent Since Theorem \ref{thm:planeUnion} in this setting implies
 \begin{equation*}
     \dim_P(F)\geq n-1 + \frac{t}{n}, 
 \end{equation*}
 Theorem \ref{thm:unionHyperplanes} is an improvement for all $n$ and $0<t<n$. It is worth comparing this result to Corollary 1.4 in \cite{bushling2025extension}, which established an essentially optimal packing dimension bound for certain ``spread'' Furstenberg-type sets in $\mathbb{R}^2$. That is, Corollary 1.4 assumed $\mathcal{P}$ was large as a subset of the Grassmannian, not the affine Grassmannian. Hence, these results do not overlap even in the case $n=2$.  

\subsection{Overview of the paper}

The proofs of our theorems rely on tools from algorithmic information theory, specifically effective dimension. These tools have recently been used to make progress on a number of problems in geometric measure theory \cite{altaf2023distance, lutz2020bounding, stull2022pinned}. By the point-to-set principle of J. Lutz and N. Lutz, bounds on the information content of points in a set imply corresponding bounds on the classical dimension of the set itself \cite{lutz2018algorithmic}. So, for instance, when we show that any point in the $k$-plane extension $F$ of some set $E$ is easily computed given two points in $E$ (and hence that the information content of points in $F$ cannot be too large compared to the information content of points in $E$), this implies that the \emph{dimension} of $F$ cannot be too much larger than the \emph{dimension} of $E$. 

We briefly introduce the necessary information-theoretic preliminaries for points in $\mathbb{R}^n$ in Section 2. Roughly speaking, these preliminaries (usually) turn out to be a sufficient toolkit for our new ``extension'' results. However, because our ``union'' results directly invoke the packing dimension of sets of $k$-planes, it is necessary to develop a notion of effective dimension for the Grassmannian and the affine Grassmannian. This is the goal of Section 3, which unfortunately is necessarily somewhat technical. In particular, we have to define a number of Turing machines suited to the various proofs. However, we outline Section 3 below so that readers can elect to avoid details that are not germane to them, or even to skip over the majority of the section on a first read. 

Section 3 is divided into four subsections. In the first, we define the complexity of elements of the (non-affine) Grassmannian by associating to each $k$-plane through the origin the corresponding orthogonal projection matrix.\footnote{Identifying elements of the Grassmannian with orthogonal projection matrices as a basis to perform practical computations can be suboptimal; in particular, this approach suffers from numerical instability (c.f. \cite{lai2020simplergrassmannianoptimization}). However, this definition is viable for our purposes for a variety of reasons, among them (1) working with rational projection matrices means our Turing machines will typically be performing exact as opposed to approximate computations (2) there is no time or other resource bound on these machines as they perform their computations (3) even a massive fixed loss in precision will not affect our results, since typically the \emph{last} step of our arguments in later sections will be to let the precision parameter go to infinity.} The definition of effective dimension in $\mathbb{R}^n$ uses the complexity of rational points that suitably approximate an arbitrary point, and likewise we proceed on the Grassmannian via ``rational'' projection matrices. The Grassmannian is a metric space, so with our setup, we are able to apply a result of J. Lutz, N. Lutz, and Mayordomo to obtain a point-to-set principle \cite{LuLuMay2023PtS}. 

The second subsection extends the machinery of Section 3.1 from the Grassmannian to the affine Grassmannian by identifying $\mathcal{A}(n, k)$ with $\mathcal{G}(n, k)\times \mathbb{R}^{n- k}$ in a manner that is ``comprehensible'' to our Turing machines and again leads to a point-to-set principle. The third subsection has one overarching goal: showing that the symmetry of information (Proposition \ref{prop:symmetry}) holds in the context of $k$-planes. 

The final subsection is likely the most important for readers, since it is where we collect a number of ``geometric'' lemmas which will be directly employed in the proofs of several of our theorems. However, even in this section, readers are likely to find the particulars of the proofs to be considerably less important than the statements of the lemmas. 

We highlight two recurring ideas that feature in the proofs of some of our main results. 
\begin{itemize}
\item Lemma \ref{lem:pointsDeterminePlanes} -- the most important geometric lemma of Section 3 -- makes rigorous the notion that one should be able to computably approximate a $k$-plane given approximations of $k+1$ points that determine a basis for that plane. This lemma is used throughout Section 4. 
\item Whenever we have the freedom to do so, we will pick the points in our proofs to contain a large amount of information that is unrelated to whatever object we ultimately want to bound the complexity of. The proof of Theorem \ref{thm:basicExtension} gives is a simple example of this idea, but it is used throughout Section 4 and Section 5. 
\end{itemize}

Finally, the goal of Section 6 is the proof of Theorem \ref{thm:optimalExtension} and Theorem \ref{thm:unionHyperplanes}. Both follow from Proposition \ref{prop:everyPrecisionBound}, but this bound requires a new geometric ``intersection'' lemma that works in higher dimensions: Lemma \ref{lem:intersectionGeo}. The proof of this lemma involves the application of a number of Section 3's results to reduce the problem to the two-dimensional case where it is already known; see Lemma 7 of \cite{lutz2020bounding}.

Aside from the role of Lemma \ref{lem:intersectionGeo}, the deduction of Theorem \ref{thm:optimalExtension} is essentially the same as the proof of Theorem 1.3 in Bushling and the author's previous work \cite{bushling2025extension}. However, the proof of Theorem \ref{thm:unionHyperplanes} requires several new ideas as compared to Corollary 1.4 in the same paper. Informally, Proposition \ref{prop:everyPrecisionBound} is especially well-suited to dealing with one dimensional objects like directions $\mathbb{R}^2$, which is the content of Corollary 1.4. For higher-dimensional objects analyzed at arbitrary precisions, it often performs terribly. Hence, the proof of Theorem \ref{thm:unionHyperplanes} requires close study of the complexity of a given hyperplane $(a, b)$ and the careful selection of a suitable sequence of precisions.

\section{Preliminaries}
We provide a brief overview of some of the information theoretic preliminaries; more detailed discussions can be found in works including \cite{lutz2018algorithmic, lutz2020bounding, stull2022pinned}. We also recommend \cite{downey2010} for readers interested in a more thorough treatment of Kolmogorov complexity. 

Let $K^A(\sigma)$ denote the (prefix-free) Kolmogorov complexity of a string $\sigma$ (relative to an oracle $A\subseteq\mathbb{R}^n$), and let $K^A(\sigma\mid\tau)$ denote the Kolmogorov complexity of $\sigma$ given $\tau$. After encoding rational points as binary strings in some fixed manner, we may define the Kolmogorov complexity at precision $r$ of $x\in\mathbb{R}^n$ as
\begin{equation*}
K^A_r(x)=\min\left\{K(p)\,:\,p\in B_{2^{-r}}(x)\cap\Q^m\right\},
\end{equation*}
and the Kolmogorov complexity of $x$ at precision $r$ given $y$ at precision $s$ as 
\begin{equation*}
K^A_{r,s}(x\mid y)=\max\big\{\min\{K_r(p\mid q)\,:\,p\in B_{2^{-r}}(x)\cap\Q^m\}\,:\,q\in B_{2^{-s}}(y)\cap\Q^n\big\}.
\end{equation*}
These are the standard definitions of Kolmogorov complexity at precision $r$, but N. Lutz and Stull showed \cite{lutz2020bounding} that these definitions also agree up to a logarithmic term with 
\begin{equation*}
    K^A(x\upharpoonright r) \qquad \text{and} \qquad  K^A(x\upharpoonright r \mid y\upharpoonright s),
\end{equation*}
respectively. Here, $x\upharpoonright r$ represents the truncation of the binary expansion of $x$ to $r$ digits. The second approximate equality is often especially useful, as the definition of conditional complexity involves several steps and can be unwieldy. 

The next inequalities express that oracle access to the binary expansion of $y$ is at least as helpful as precision $s$ access to $y$ in computing $x$. Precision $s$ access to $y$, in turn, cannot appreciably increase the difficulty of a computation. 
\begin{equation*}
    K_r^{A, y}(x)\leq K_{r, s}^A(x\mid y) + O(\log(r + s)) \leq K_r^A(x) + O(\log (r + s)).
\end{equation*}
A key property of Kolmogorov complexity is the symmetry of information, which N. Lutz and Stull proved that also holds for complexity at precision $r$, as expressed in the following lemma \cite{lutz2020bounding}. 
\begin{lem}\label{lem:soIBasic}
	For every $m,n\in\N$, $x\in\R^m$, $y\in\R^n$, and $r,s\in\N$ with $r\geq s$,

		\begin{equation*}
        \vert K^A_{r, s}(x\mid y)+K^A_s(y)-K^A_{r, s}(x,y)\big\vert \leq O(\log r)+O(\log\log \vert y\vert).
        \end{equation*}
\end{lem}
In practice, we will often collapse error terms involving the norm of points into the implicit constants in inequalities. We will extend Lemma \ref{lem:soIBasic} to the Grassmannian in Proposition \ref{prop:symmetry}, which is a major goal of the next section.

A final property of complexity at precision $r$ that we mention is a bound of Case and J. Lutz \cite{case2015dimension} that limits how much $K_r^A(x)$ can increase on intervals of precisions. 
\begin{prop}\label{prop:caseLutz}
For any $A \subseteq \mathbb{N}$, $r,s \in \N$, and $x \in \R^n$,
\begin{equation*}
    K_{r+s}^A(x) \leq K_r^A(x) + ns + O(\log(s + r)).
\end{equation*}
\end{prop}
\noindent We will frequently apply this bound in situations where $s$ is essentially constant with respect to $r$, meaning we only incur a logarithmic error moving between precisions. 

Finally, we explain the connection between Kolmogorov complexity and classical dimension. J. Lutz developed a notion of effective fractal dimension using computable betting strategies called gales \cite{Lutz03a, Lutz03b}. Mayordomo proved the equivalence of another definition involving Kolmogorov complexity \cite{Mayordomo02}, which we will utilize in this work. In particular, define
\begin{equation*}
    \dim^A(x):= \liminf_{r\to\infty} \frac{K_r^A(x)}{r} \qquad \text{and} \qquad \Dim^A(x):= \limsup_{r\to\infty}  \frac{K_r^A(x)}{r}.
\end{equation*}
We refer to the first of these as the effective Hausdorff dimension of $x$ relative to $A$ and the second as the effective packing dimension of $x$ relative to $A$. The point-to-set principle of N. Lutz and J. Lutz makes precise the connection between these effective notions and the corresponding classical fractal dimensions \cite{lutz2018algorithmic}. For any $E\subseteq\mathbb{R}^n$, 
\begin{equation*}
    \dim_H(E)=\min_{A\subseteq\mathbb{N}}\sup_{x\in E}\dim^A(x) \qquad \text{and} \qquad \dim_P(E)=\min_{A\subseteq\mathbb{N}}\sup_{x\in E}\Dim^A(x).
\end{equation*}

\section{The Grassmannian and effective dimension}

In order to apply the techniques of algorithmic information theory to $k$-planes, we first need a notion of the complexity of a $k$-plane. In building this notion, there will be two main considerations. On one hand, it is desirable that we work in the framework of \cite{LuLuMay2023PtS}, which established versions of the point-to-set principle in the very general setting of separable metric spaces -- a necessary tool as we study the size of collections of $k$-planes. On the other hand, we need to make our definitions in such a way that Turing machines can ``understand'' the geometry of the $k$-planes, in particular so that we can prove the lemmas of Section 3.4. As mentioned in the introduction, it will be crucial that our notion allows us to use appropriate points on a $k$-plane to determine the $k$-plane, up to some error. Additionally, we would like to be able to compute, say, the orthogonal complement of a given $k$-planes or points that lie on it. We will make this idea more precise with a series of geometric lemmas at the end of this section.  

As a note for readers before proceeding with this section in earnest, its proofs will vary considerably. In particular, some of the proofs in Section 3.1 and Section 3.2 are very straightforward from a computability-theoretic standpoint. However, some readers will likely come from a different background, and we are unaware of an explicit treatment of much of this content. Hence, we typically choose to include more rather than less detail, without claiming that every detail is especially interesting. Reiterating advice from the introduction, we expect that many readers would benefit from skipping to Section 3.4 where the results -- building on the rest of Section 3 -- are perhaps more interesting and in any event are more immediately applicable to our main theorems. 

\subsection{The Grassmannian}
To start, let $\mathcal{G}(n, k)$ denote the Grassmannian, that is, the set of $k$-planes in $\mathbb{R}^n$ that contain the origin. Note that each such plane corresponds to a unique linear operator: orthogonal projection onto that $k$-plane. For $V\in \mathcal{G}(n, k)$, we denote the projection operator $\pi_V$. Hence, we can regard $\mathcal{G}(n, k)$ as a subset of the space of $n\times n$ matrices. We will sometimes use expressions like $Vx$ instead of $\pi_Vx$ in the context of matrix multiplication. 

Associating to each $k$-plane its projection operator also allows us to define a metric on the Grassmannian. For $V_1, V_2\in \mathcal{G}(n, k)$, let
\begin{equation*}
    \rho (V_1, V_2) = \sup_{x\in\mathbb{R}^{n}, \vert x\vert =1}\vert \pi_{V_1}x - \pi_{V_2}x\vert. 
\end{equation*}
Any time we refer to $\mathcal{G}(n, k)$ in the sequel, it will be assumed it is equipped with this metric. Importantly, this metric plays well with other notions of distance between planes. Define the ``maximum distance'' between $V_1\in \mathcal{G}(n, k_1)$ and $ V_2\in \mathcal{G}(n, k_2)$ to be

\begin{equation*}
    m(V_1, V_2) = \sup_{v_1\in V_1,\vert v_1\vert = 1} \inf_{v_2\in V_2, \vert v_2\vert = 1} \vert v_1 - v_2 \vert.
\end{equation*}

It is straightforward to observe that $m$ is a metric when $k_1=k_2$ and that it is equivalent to $\rho$ in the sense that there exists some absolute constant $C_{n, k}$ such that for every $V_1, V_2\in\mathcal{G}(n, k)$,

\begin{equation*}
\frac{1}{C_{n, k}} m(V_1, V_2) \leq \rho(V_1, V_2) \leq C_{n, k} m(V_1, V_2).
\end{equation*}

Recall that the singular values of a real matrix $B$ are the square roots of the eigenvalues of $B^\text{T}B$. Another interpretation is that the singular values are the lengths of the axes of the image of the unit ball under multiplication by $B$. At several points, we will need the following geometric fact,

\begin{lem}\label{lem:singularValue}
    Let $\ve>0$ be given. Assume $v_1, ..., v_k$ is a basis for $V\in\mathcal{G}(n, k)$. Let $v^\prime_i$ be vectors such that $\vert v^\prime_i - v_i\vert <\ve $ and let $V^\prime$ be the space spanned by $v^\prime_1, ..., v^\prime_k$. Then

    \begin{equation*}
    m(V, V^\prime)\leq n \sigma^{-1} \ve,   
    \end{equation*}
    where $\sigma$ is the smallest nonzero singular value of the column matrix formed by $v_1, ..., v_k$.
\end{lem}
\begin{proof}
Let $x$ be any point on $V$ with norm $1$. $v_1, ..., v_k$ is a basis for $V$, so $x$ can be written uniquely as $x = \sum_{i=1}^k\alpha_i v_i$. Define $x^\prime = \sum_{i=1}^k\alpha_i v^\prime_i$; clearly $x^\prime\in V^\prime$. Then
\begin{equation*}
    \vert x - x^\prime\vert  \leq \ve \sum_{i=1}^k\vert \alpha_i\vert. 
\end{equation*}
Hence, 
\begin{equation*}
    m(V, V^\prime)\leq \ve \sup_{\vert x\vert = 1}\sum_{i=1}^k\vert \alpha_i\vert.
\end{equation*}
Let $A$ denote the matrix with $v_1, ..., v_k$ as its columns and define $\sigma$ to be the smallest nonzero singular value of $A$. Note that $\min\{\vert Ay\vert : \vert y\vert =1\}=\sigma$. Let $x$ of norm 1 be given and let $\alpha$ denote the coordinates of $x$ in the basis $v_1, ..., v_k$. Since $\frac{\alpha}{\vert \alpha \vert}$ has norm $1$, $x$ has norm $1$, and $A\frac{\alpha}{\vert \alpha \vert} = \frac{x}{\vert\alpha \vert} $, we have that  $\max_i\vert \alpha_i\vert\leq \vert \alpha \vert \leq \sigma^{-1}$. Thus, 
\begin{equation*}
    \sup_{\vert x\vert = 1}\sum_{i=1}^k\vert \alpha_i\vert\leq n \sigma^{-1}.
\end{equation*}
and the conclusion follows. 
    
\end{proof}

Let $\mathbb{Q}^{n\times n}$ denote the space of $n\times n$ matrices such that all entries are rational. Fix some standard, computable encoding of the elements of $\mathbb{Q}^{n\times n}$ as natural numbers, and call this encoding $h$. Passing through $h$, it is easy to define a computable encoding of the elements of $\mathbb{Q}^{n\times n}$ as finite binary strings. 

At this point, it is worth taking note of a few things Turing machines can do with these rational matrices. These observations are likely quite obvious to readers with some computability-theoretic background, but we elect to explicitly mention them in keeping with our reasoning at the start of this section. 
\begin{itemize}
    \item Turing machines can calculate sums, products, transposes, and traces of rational matrices. 
    \item Turing machines can calculate the inverse of an invertible rational matrix $B$ by sequentially testing rational matrices $Q$ to see whether $QB=BQ=I$. 
    \item Turing machines can calculate the characteristic polynomial of an $n\times n$ rational matrix (by using any procedure for computing determinants), in the sense that they can output the exact coefficients of the polynomial (which have to be rational). 
    \item Using the characteristic polynomial, Turing machines can decide the rational eigenvalues of a matrix, including their algebraic multiplicity. By employing zero-finding algorithms, Turing machines can also approximate irrational eigenvalues to a desired precision. 
    \item As a consequence of the above, Turing machines can determine the rank of a rational matrix. 
\end{itemize}
These observations illustrate some of the advantage of working with rational matrices. Turing machines can perform basic computations involving rational numbers that can be strung together into calculations involving rational matrices, whereas in general it is not even decidable whether two merely computable numbers are even \emph{equal}. 

With these tools in hand and recalling the connection between $k$-planes and projections, we can move forward with our effectivization of the Grassmannian. $\mathcal{G}(n, k)$ is exactly the set of $n\times n$ matrices $V$ such that $V^2=V=V^{\text{T}}$ and $\text{trace}(V)=k$. Hence, whether a given $n\times n$ matrix consisting only of rational entries is in the Grassmannian is decidable. Our encoding $f_G$ will be as follows: given a finite binary string $\sigma$, convert it to a natural number $m$ in some standard way. Using the ordering of $h$, sequentially check whether the $i$th rational $n\times n$ matrix is in $\mathcal{G}(n, k)$. $f_G(\sigma)$ is the $m$th matrix verified to be in $\mathcal{G}(n, k)$. Note that the inverse of $f_G$ is also computable: given $Q\in \mathcal{G}(n, k)\cap \mathbb{Q}^{n\times n}$, simply keep checking the matrices in order until one finds $Q$. 

Equipped with such an encoding $f_G$, we can define the complexity of rational matrices $Q$:

\begin{equation*}
    K^A(Q) = \min\{K^A(\sigma): f_G(\sigma)=Q\}.
\end{equation*}

\noindent The following lemma shows that working with rational matrices will be ``robust'', in that planes spanned by rational vectors are given by rational projection matrices. 

\begin{lem}\label{lem:rationalColumns}
    Suppose $Q\in \mathcal{G}(n, k)$ has a basis of vectors $q_1, ..., q_k\in \mathbb{Q}^n$. Then $Q\in \mathbb{Q}^{n\times n}$.
\end{lem}

\begin{proof}
    For orthogonal projections onto the subspace given by the range of a real matrix $A$, we have the following formula:
    \begin{equation*}
        \pi_A = A\left(A^{\text{T}}A\right)^{-1}A^{\text{T}}.
    \end{equation*}
    Setting the columns of $A$ equal to $q_1, ..., q_k$ and noting that the transposes, products, and inverses of rational matrices are rational suffices. 
\end{proof}

The next lemma justifies that, as in the Euclidean case, we will be able to define the complexity at a given precision of a $k$-plane -- using the now known complexity of ``rational'' $k$-planes -- via approximation

\begin{lem}\label{lem:densityG}
    $\mathbb{Q}^{n\times n}\cap \mathcal{G}(n, k)$ is dense in $\mathcal{G}(n, k)$.
\end{lem}
\begin{proof}
Let $V\in\mathcal{G}(n, k)$ and $\delta>0$ be given. Suppose $v_1, ..., v_k$ is a basis for $V$ and (as in Lemma \ref{lem:singularValue}), let $\sigma$ denote the smallest nonzero singular value of the corresponding column matrix. Let $v^\prime_1, ..., v^\prime_k$ be linearly independent rational vectors such that $\vert v^\prime_i - v_i\vert < \frac{\sigma \delta}{n}$. By Lemma \ref{lem:rationalColumns}, the unique $V^\prime\in \mathcal{G}(n, k)$ spanned by these vectors is in $\mathbb{Q}^{n\times n}\cap \mathcal{G}(n, k)$, and by Lemma \ref{lem:singularValue}, $m(V, V^\prime)<\delta$. Since $m$ and $\rho$ are equivalent as metrics, this completes the proof.

\end{proof}

In fact, more is true of this dense subset; it testifies to the computability of the metric space $\mathcal{G}(n, k)$. In particular, matching \cite{LuLuMay2023PtS}, we say a metric space is computable\footnote{Note that this term is occasionally used in a slightly different sense; see e.g. \cite{weihrauch2000, lacombe1959metric}} if there is some computable function $g:\left(\{0, 1\}^*\right)^2\times \mathbb{Q}^{+} \to \mathbb{Q}$ such that for every $v, w$, and $\delta$, 
\begin{equation*}
    \vert g(v, w, \delta) - \rho(f(v), f(w))\vert <\delta.
\end{equation*}
Intuitively, this says that one can approximate the distance between $k$-planes from the strings encoding them to any given accuracy.  

\begin{lem}\label{lem:GComputable}
$\mathcal{G}(n, k)$ is computable. 
\end{lem}
\begin{proof}
    We describe such a function $g$ via the corresponding Turing machine $M_g$. First, $M_g$ uses $f_G$ to obtain the projection matrices $f_G(v)$ and $f_G(w)$. Then, $M_g$ calculates 
\begin{equation*}\delta^\prime = \max_{i\in \mathbb{Z}}\{\frac{1}{2^i}: \frac{1}{2^i}<\frac{\delta}{n} \text{ and } \frac{1}{2^i} < \frac{1}{8n}\} 
\end{equation*}
It suffices to check that $\vert g(v, w, \delta) - \rho(f_G(v), f_G(w))\vert <\delta^\prime$.

Next, $M_g$ generates all of the (finitely many) $(\delta^\prime)^3$-dyadic rational vectors $u\in\mathbb{R}^n$ satisfying $1 -(\delta^\prime)^2 \leq \vert u\vert \leq 1 +(\delta^\prime)^2$; call this set $\mathcal{U}$. It is easy to see that membership in $\mathcal{U}$ is decidable. For any rational vector $u$ and rational number $q$, a Turing machine can decide whether $\vert u\vert \leq q$ by first checking whether $\vert v\vert^2=q^2$; a computation only involving rational numbers. If these are not equal, then either $\vert u\vert<q$ or $\vert u\vert>q$, and better and better rational approximations of $\vert u\vert$ will eventually indicate which is the case. 

Using the projection matrices, $M_g$ calculates 
\begin{equation*}
m=\max_{u\in \mathcal{U}} \vert \pi_{f_G(v)}u - \pi_{f_G(w)}u\vert 
\end{equation*}
and outputs some rational approximation $q$ of $m$, specifically some $q$ such that

\begin{equation}\label{eq:rationalApprox}
m - \frac{\delta^\prime}{2}<q< m.
\end{equation}

Since $\mathcal{S}^{n-1}$ is compact, we may choose $u^\prime$ to be a vector of magnitude $1$ such that
\begin{equation}\label{eq:compactEquality}
 \rho(f_G(v), f_G(w))=  \vert \pi_{f_G(v)} u^\prime - \pi_{f_G(w)} u^\prime \vert.
\end{equation}
By the definition of $\mathcal{U}$, there is some $u^{\prime\prime}\in \mathcal{U}$ and $d\in \mathbb{R}^n$ of magnitude at most $(\delta^\prime)^2$ such that $u^\prime = u^{\prime\prime} + d$. Projections are Lipschitz with constant $1$, hence,
\begin{equation}\label{eq:lipApprox}
\left\vert \vert \pi_{f_G(v)} u^\prime - \pi_{f_G(w)} u^\prime \vert - \vert \pi_{f_G(v)} u^{\prime\prime} - \pi_{f_G(w)} u^{\prime\prime} \vert \right\vert \leq 2 (\delta^\prime)^2<\frac{\delta^\prime}{2}.
\end{equation}

Combining \eqref{eq:rationalApprox}, \eqref{eq:compactEquality}, and \eqref{eq:lipApprox} with the triangle inequality completes the proof. 

\end{proof}

\noindent We now introduce the complexity at precision $r$ of elements of the Grassmannian. Let $V\in G(n, k)$. We define

\begin{equation*}
    K_r^A(V) = \min\{K^A(Q): Q\in \mathcal{G}(n, k)\cap \mathbb{Q}^{n\times n}\cap B_{2^{-r}}(V)\},
\end{equation*}

\noindent where the ball is with respect to the metric $\rho$.

This immediately gives us corresponding notions of effective dimension 

\begin{equation*}
    \dim^A(V) = \liminf_{r\to\infty}\frac{K^A_r(V)}{r} \qquad \text{ and } \qquad  \Dim^A(V) = \limsup_{r\to\infty}\frac{K^A_r(V)}{r}.
\end{equation*}

Applying Theorem 4.1 and Theorem 4.2 of \cite{LuLuMay2023PtS} (with the canonical gauge family $\theta_s(\delta) = \delta^s$) gives point-to-set principles on the Grassmannian.

\begin{prop}
For every $\mathcal{V}\subseteq \mathcal{G}(n, k)$, 
\begin{equation*}
    \dim_H(\mathcal{V}) = \min_{A\subseteq\mathbb{N}}\sup_{V\in\mathcal{V}}\dim^A(V) \qquad \text{and} \qquad \dim_P(\mathcal{V}) = \min_{A\subseteq\mathbb{N}}\sup_{V\in\mathcal{V}}\Dim^A(V).
\end{equation*}

\end{prop}

\subsection{The affine Grassmannian}

At this point, we understand the information content of $k$-planes passing through the origin. However, clearly not every $k$-plane has this form. The set of translates of elements of the Grassmannian is the affine Grassmannian, which we denote $\mathcal{A}(n, k)$. Note that every $P\in \mathcal{A}(n, k)$ can be uniquely representative by a sum $V + t$ where $V\in \mathcal{G}(n, k)$ and $t\in V^\perp$. Hence, this space has dimension $(k+1)(n-k)$. 

This representation suggests a metric for $\mathcal{A}(n, k)$. For $P_1, P_2\in A(n, k)$, write $P_1 = V_1 + t_1$ and $P_2 = V_2 + t_2$. Then,
\begin{equation*}
    \rho(P_1, P_2) = \rho(V_1, V_2) + \vert t_1 - t_2\vert. 
\end{equation*}
This abuse of notation $\rho$ is reasonable because the restriction of this metric to the Grassmannian gives the same metric as in the previous subsection. 

$V^\perp$ is isomorphic to $\mathbb{R}^{n-k}$, but it is necessary to define a basis; our Turing machine does not necessarily understand $t\in V^\perp$ automatically. In particular, we would like our basis to be both computable from $V$ and not too singular. The following lemma shows that (a precise version of) the second condition is decidable. 

\begin{lem}\label{lem:niceComputableBasis}
     There exists a Turing machine $M$ that, given $q_1, q_2, ..., q_k\in\mathbb{Q}^n$, decides whether all of the following are satisfied:
     \begin{enumerate}
         \item Every vector $q_i$ satisfies $\frac{1}{2}\leq\vert q\vert \leq 2$,
         \item The vectors are linearly independent,     
        \item The smallest singular value of the column matrix is at least $\frac{1}{2}$.
     \end{enumerate}
\end{lem}
\begin{proof}
    It is easy to see that $M$ can decide whether the first condition is satisfied. Checking the second and third conditions both involve $M$ calculating the singular values of the matrix $A$ that has columns $q_i$. The $k$ vectors are linearly independent if and only if this matrix has rank $k$, and we have already seen that the rank is computable. 

    $M$ can determine the characteristic polynomial, so it is capable of checking the multiplicity of $\frac{1}{4}$ as a root of this polynomial (which corresponds to $\frac{1}{2}$ as a singular value). Call this multiplicity $j$. $M$ then more and more closely approximates the roots of the characteristic polynomial until its approximations of $k-j$ of the roots are inconsistent with $\frac{1}{4}$; if all of these are more than $\frac{1}{4}$, the third condition is satisfied. 
\end{proof}

With this lemma in hand, it is easy to describe a Turing machine $T$ that generates a basis for $V^\perp$ from $V$. Fix some computable enumeration of sets of $n-k$ vectors in $\mathbb{Q}^n$. For any set $q_1, ..., q_{n-k}$ in this enumeration, $T$ calculates the component of each vector in $V^\perp$, that is, $q_i -Vq_i$, which is clearly in $V^\perp \cap \mathbb{Q}^n$. Then, $T$ uses $M$ to determine whether the conditions of Lemma \ref{lem:niceComputableBasis} are satisfied for the new collection of $n-k$ vectors. As soon as it finds that the conditions are satisfied for some collection of $n-k$ vectors, it halts and outputs them. It is essentially a consequence of Lemma \ref{lem:singularValue} and the existence of orthonormal bases that for any $k$-dimensional subspace, there exists some rational basis of $V^\perp$ satisfying the conditions, so $T$ is guaranteed to halt. 

Denoting these basis vectors $e_1(V), ..., e_{n-k}(V)$, we can now identify $\mathcal{A}(n, k)$ with $\mathcal{G}(n, k)\times \mathbb{R}^{n-k}$. In particular, $(V, v_1, ..., v_{n-k}) := V + \sum_{i=1}^{n-k} v_i e_i(V).$ As such, 
\begin{equation*}
    \rho\left((V, v_1, ..., v_{n-k}),(W, w_1, ..., w_{n-k})\right) = \rho(V, W) + \left\vert \sum_{i=1}^{n-k} v_i e_i(V) - \sum_{i=1}^{n-k} w_i e_i(W)\right\vert.
\end{equation*}

As before, we also need a notion of the ``rational'' elements of $\mathcal{A}(n, k)$; these will be the planes in $\left(\mathcal{G}(n, k)\cap \mathbb{Q}^{n\times n}\right)\times \mathbb{Q}^{n-k}$. There exists a computable encoding of the elements of $\left(\mathcal{G}(n, k)\cap \mathbb{Q}^{n\times n}\right)\times \mathbb{Q}^{n-k}$ as binary strings. We can construct one using the encoding of the Grassmannian, $f_G$, some standard encoding of $\mathbb{Q}^{n-k}$, and the standard pairing function. Denote our new encoding $f_A$. We define the complexity of the rational elements of the affine Grassmannian using $f_A$,

\begin{equation*}
    K^A(Q) = \min\{K^A(\sigma): f_A(\sigma)=Q\}.
\end{equation*}

\noindent Again, these rational elements form a dense subset.

\begin{lem}\label{lem:densityA}
    $\mathcal{G}(n, k)\cap \mathbb{Q}^{n\times n}\times \mathbb{Q}^{n-k}$ is dense in $\mathcal{A}(n ,k)$.
\end{lem}
\begin{proof}
    This follows immediately from Lemma \ref{lem:densityG}, the density of $\mathbb{Q}^{n-k}$ in $V^\perp$, and the definition of the metric $\rho$. 
\end{proof}

Complexity at precision $r$ in the affine Grassmannian is defined in essentially the same way as it was for the Grassmannian. 

\begin{equation*}
    K_r^A(P) = \min\{K^A(Q): Q\in \left(\mathcal{G}(n, k)\cap \mathbb{Q}^{n\times n}\times \mathbb{Q}^{n-k}\right)\cap B_{2^{-r}}(P)\},
\end{equation*}

\begin{lem}\label{lem:AComputable}
    $\mathcal{A}(n, k)$ is computable.
\end{lem}
\begin{proof}
    Let $f_Q$ denote the encoding of elements of $\mathbb{Q}^{n-k}$ as finite binary strings that was (implicitly) used in the definition of $f_A$. Given strings $v$ and $w$, our Turing machine can first compute strings $v_1, v_2$ and $w_1, w_2$ such that $f_A(v)= \left(f_G(v_1), f_Q(v_2)\right)$ and $f_A(w)= \left(f_G(w_1), f_Q(w_2)\right)$. In view of the definition of the metric $\rho$, and the proof of Lemma \ref{lem:GComputable} (which handles the distance in the ``first coordinate''), it suffices to construct some Turing machine $M_g$ such that the corresponding function $g$ satisfies 
    \begin{equation*}
    \left\vert g(v, w, \delta) - \left\vert \sum_{i=1}^{n-k} \pi_i (f_Q(v_2)) e_i(f_G(v_1)) - \sum_{i=1}^{n-k} \pi_i (f_Q(w_2)) e_i(f_G(w_1))\right\vert \right\vert <\frac{\delta}{2}.
    \end{equation*}
    In particular, $M_g$ already has access to $v_1, v_2$ and $w_1, w_2$, so it computes the image of these under $f_Q$ and $f_G$ as appropriate. Since our bases for $f_G(v_2)^\perp$ and $f_G(w_2)^\perp$ are computable, at this point the remaining problem amounts to approximating the norm of a known rational vector to a given level of accuracy, which can certainly be accomplished by some Turing machine. 
\end{proof}

Again applying \cite{LuLuMay2023PtS}, we obtain the following point-to-set principle.

\begin{prop}
For every $\mathcal{P}\subseteq \mathcal{A}(n, k)$, 
\begin{equation*}
    \dim_H(\mathcal{P}) = \min_{A\subseteq\mathbb{N}}\sup_{P\in\mathcal{P}}\dim^A(P) \qquad \text{and} \qquad \dim_P(\mathcal{P}) = \min_{A\subseteq\mathbb{N}}\sup_{P\in\mathcal{P}}\Dim^A(P).
\end{equation*}

\end{prop}

\subsection{Conditional complexity and the symmetry of information}

Before we begin this section in earnest, we describe some notation. In order to freely work with the spaces $\mathcal{G}(n_1, k_1)$, $\mathcal{A}(n_2, k_2)$, and $\mathbb{R}^{n_3}$, we will use $\mathcal{X}$ and $\mathcal{Y}$ to denote \emph{any} of the spaces for any choice of $n_i$ and $k_i$. In particular, we allow the possibility that $\mathcal{X}=\mathcal{Y}$, and we will use $X$ and $Y$ to denote elements of $\mathcal{X}$ and $\mathcal{Y}$  respectively. Between the preliminaries and the previous two subsections, we have seen the rational elements of each kind of space; $\mathbb{Q}_{\mathcal{X}}$ will be shorthand for the rational elements of the space $\mathcal{X}$. Likewise $f_\mathcal{X}$ denotes the encoding for the space $\mathcal{X}$. 

The preliminaries included a definition of conditional complexity for points in Euclidean space, which we slightly generalize. Let $X\in\mathcal{X}$ and $Y\in\mathcal{Y}$ be given. Let $r, s$ be precisions and $A\subseteq\mathbb{N}$ be an oracle. Then we define
\begin{equation*}
    K^A_{r, s}(X\mid Y) = \max_{P\in B_{2^{-s}}(Y)\cap\mathbb{Q}_{\mathcal{Y}}}\left(\min_{Q\in B_{2^{-r}}(X)\cap\mathbb{Q}_{\mathcal{X}}} K^A(Q \mid P)\right).
\end{equation*}

In addition to conditional complexity, it is also necessary to define the complexity of tuples of such objects. This is slightly less immediate than in the case of Euclidean space. If $q\in\mathbb{Q}^n$ and $p\in\mathbb{Q}^m$, then $(q, p)\in\mathbb{Q}^{n+m}$, but it is not as obvious what larger space the pair of, say, an element of the Grassmannian and an element of the affine Grassmannian lies in. Nevertheless, we can make the following definition for $X\in\mathcal{X}$ and $Y\in\mathcal{Y}$. 
\begin{equation*}
K^A_{r, s}(X, Y) = \min\left\{K^A(\langle\sigma, \tau \rangle): f_{\mathcal{X}}(\sigma)\in B_{2^{-r}} \text{ and } f_{\mathcal{Y}}(\tau)\in B_{2^{-s}}(Y)\right\},
\end{equation*}
where $\langle, \rangle$ denotes a standard pairing function. We can define the complexity of $n$-tuples of such objects similarly, and it is straightforward that this definition agrees with the usual one for $X, Y$ in Euclidean space (up to at most a logarithmic error term in the precisions). 

We are now able to state the symmetry of information. 

\begin{prop}\label{prop:symmetry}
    Let $X\in\mathcal{X}$ and $Y\in\mathcal{Y}$. For every $A\subseteq\mathbb{N}$ and precisions $r, s$, 
    \begin{equation*}
        K^A_{r, s}(X, Y) = K^A_s(Y) + K_{r, s}^A(X\mid Y) \pm O(\log (r + s)).
    \end{equation*}
\end{prop}

The proof we employ is somewhat more complex than in Euclidean space. This is because in that setting the usual definition of the complexity of a point (in terms of $K$-minimizing rationals at precision $r$) is equal to the complexity of the string formed by truncating the point to its first $r$ binary digits. In other words, there is a dyadic rational associated to each point and precision, and we can certainly apply the usual symmetry of information to the strings that encode these dyadic rationals. We have not defined any notion of dyadic rationals in the Grassmannian, and indeed it is less obvious how one should go about this as compared to $\mathbb{R}^n$.\footnote{One possibility would be to treat the projection matrix as an element of $\mathbb{R}^{n\times n}$ and then find the corresponding $2^{-r}$ dyadic rational, but this new matrix is not necessarily a projection matrix. Another possibility would be to give explicit charts for the Grassmannian or affine Grassmannian (as manifolds) and then declare that an element of either is $2^{-r}$-dyadic if its preimage in one of the charts is. For our purposes, it will be simplest to somewhat arbitrarily define a set of rationals elements that share key properties with dyadic rationals in $\mathbb{Q}^n$} 

However, we can define a set of rational elements for each space that serves the same purpose. In particular, we will be able associate an element in $\mathbb{Q}_{\mathcal{X}}$ to each pair $(X, r)$, which we will denote $d(X, r)$. We will provide an explicit definition for $d(X, r)$ shortly, but the main use of this notion will be the following lemma,

\begin{lem}\label{lem:rationalsToStrings}
  Let $X$ and $Y$ be elements of any of $\mathcal{G}(n_1, k_1)$, $\mathcal{A}(n_2, k_2)$, or $\mathbb{R}^{n_3}$. For every $A\subseteq\mathbb{N}$ and precisions $r, s$, 
        \begin{enumerate}
        \item[\textup{(i)}] $K^A_{r, s}(X, Y) = K^A(d(X, r), d(Y, s))\pm O(\log(r + s))$,
        \item[\textup{(ii)}] $K^A_{s}(Y) = K^A(d(Y, s))\pm O(\log(s))$,
        \item[\textup{(iii)}] $K^A_{r, s}(X\mid Y) = K^A(d(X, r)\mid d(Y, s))\pm O(\log(r + s))$.
    \end{enumerate}
\end{lem}

The symmetry of information for the Grassmannian follows almost immediately from Lemma \ref{lem:rationalsToStrings}; we will momentarily postpone the proof of this lemma.

\begin{proofof}{Proposition \ref{prop:symmetry}}
    Since $d(X, r)$ and $d(Y, s)$ are finite-data objects, the usual symmetry of information applies and we have that 
    \begin{align*}
        K^A(d(X, r), d(Y, s)) &= K^A(d(Y, s)) + K^A(d(X, r)\mid d(Y, s)) \pm O(K^A(K^A(d(X, r), d(Y, s))))\\
        &= K^A(d(Y, s)) + K^A(d(X, r)\mid d(Y, s)) \pm O(\log (K^A(d(X, r), d(Y, s))))\\
        &= K^A(d(Y, s)) + K^A(d(X, r)\mid d(Y, s)) \pm O(\log (r+s)
        ).
    \end{align*}
    Applying the three approximate equalities of Lemma \ref{lem:rationalsToStrings} yields the desired conclusion. \qed
  \end{proofof}

It remains to rigorously define $d(X, r)$, establish some of its properties, and apply those properties in order to prove Lemma \ref{lem:rationalsToStrings}.

Let a space $\mathcal{X}$ and a precision $r$ be given. We describe a Turing machine $M$ that enumerates the elements of a set $\mathcal{D}(\mathcal{X}, r)$. $M$ starts with an enumeration of the rational elements of $\mathcal{X}$ (which we denote $\mathbb{Q}_{\mathcal{X}}$) and from this calculates an enumeration $(q_i, n_i)$ of $\mathbb{Q}_{\mathcal{X}}\times \left(\mathbb{N}\setminus \{0\}\right)$ with the property that for any $q$ and $j$, $(q, j)$ always appears before $(q, j+1)$. $M$ begins by placing $q_1$ in $\mathcal{D}(\mathcal{X}, r)$. When $M$ considers $(q_i, n_i)$, it outputs a rational approximation of the distance between $q_i$ and every one of the points already in $\mathcal{D}(\mathcal{X}, r)$ which differs by no more than $2^{-n_ir}$ from the actual distance; this is possible for $M$ because, by Lemma \ref{lem:GComputable} and Lemma \ref{lem:AComputable}, all of these metric spaces are computable. 

$M$ then checks whether any of these approximate distances is less than $2^{-r + 1} + 2^{-n_ir}$. If so, it moves onto the next element in the enumeration, if not, it adds $q_i$ to $\mathcal{D}(\mathcal{X}, r)$.

Note that this construction shows $\mathcal{D}(\mathcal{X}, r)$ is enumerable by a fixed program given only $\mathcal{X}$ and $r$. We also have the following lemma.

\begin{lem}\label{lem:uniformityForD}
   Let $l\in\mathbb{N}$ be given. For every $X\in\mathcal{X}$ and precision $r$, there exists some constant $C_{\mathcal{X}, l}$ such that for all precisions $r$ and $X\in \mathcal{X}$, 
       \begin{equation*}
           1\leq \#\{B_{2^{-(r-l)}}(X)\cap  \mathcal{D}(\mathcal{X}, r)\} \leq C_{\mathcal{X}, l}.
       \end{equation*}

\end{lem}
\begin{proof}
    To see that the lower bound holds, assume that there is no element of $\mathcal{D}(\mathcal{X}, r)$ within a distance $2^{-r}$ of $X$. By Lemma \ref{lem:densityG} and Lemma \ref{lem:densityA}, the rational elements of $\mathcal{X}$ are dense, so let $\bar{X}$ be one within $2^{-(r+1)}$ of $X$. Then $\bar{X}$ is a distance of more than $3*2^{-(r+2)}$ from any element of $\mathcal{D}(\mathcal{X}, r)$. However, if this is the case, then when $ M$ eventually checks $(\bar{X}, 3)$, it will find all of the approximate distances are at least $3*2^{-(r+2)} - 2^{-3r}> 2^{-(r+1)} + 2^{-3r}$ and add $\bar{X}$, which yields a contradiction since $\bar{X}$ is close to $X$. 

    The upper bound follows from the fact that all elements of $\mathcal{D}(\mathcal{X}, r)$ are mutually at least $2^{-(r+1)}$ apart (since $M$ checks a larger lower bound on the distances to existing elements before it adds any element). There are a number of ways this fact can be used. For instance, we can associate to $\mathcal{D}(\mathcal{X}, r)$ a packing by balls of radius $2^{-(r+2)}$. At this point, an upper bound of the desired form follows from even trivial estimates on the density of sphere packings in the Grassmannian (see, for instance, \cite{henkel2003}).
\end{proof}

In view of the second property, we can finally define $d(X, r)$ to be the first (in the order $M$ generates) $d\in\mathcal{D}(\mathcal{X}, r)$ such that the distance between $X$ and $d$ is less than $2^{-r}$. We are now in a position to complete our proof of the symmetry of information. 
  
  \begin{proofof}{Lemma \ref{lem:rationalsToStrings}}
    Recall that the lemma consists of three approximate equalities, which we prove in order.
    
    \noindent \textit{Proof of (i):} First, we show that $K^A_{r, s}(X, Y) = K^A(d(X, r), d(Y, s))\pm O(\log(r + s))$. One direction is immediate; since $d(X, r)$ and $d(Y, s)$ are rational elements lying within $2^{-r}$ of $X$ and $Y$, by definition 
    \begin{equation*}
         K^A_{r, s}(X, Y) \leq  K^A(d(X, r), d(Y, s)).
    \end{equation*}
   The other direction follows if for \emph{every} rational approximation $(\bar{X}, \bar{Y})$ of $X, Y$ within the prescribed distance, 
   \begin{equation*}
       K^A(d(X, r), d(Y, s)) \leq K^A(\bar{X}, \bar{Y}) + O(\log (r + s)),
   \end{equation*}
   where the implicit constant does not depend on $\bar{X}$ or $\bar{Y}$. We describe a Turing machine $M$ that calculates $d(X, r)$ and $d(Y, s)$ given $\bar{X}, \bar{Y}, s, r,$ strings $\sigma_{\mathcal{X}}$ and $\sigma_{\mathcal{Y}}$ encoding the spaces $\mathcal{X}$ and $\mathcal{Y}$ respectively, and natural numbers $m_\mathcal{X}$ and $m_\mathcal{Y}$. $M$ first uses $\mathcal{X}$ and $r$ to enumerate the $d_{\mathcal{X}}(i)\in \mathcal{D}(\mathcal{X}, r)$. As $M$ lists these elements, in parallel, it calculates better and better approximations of the distance between each one and $\bar{X}$, and enumerates the $d_{\mathcal{X}}(i)$ which are found to satisfy $\rho_{\mathcal{X}}(d_i, \bar{X})< 2^{-(r-1)}$. Call this refined list $\hat{d}_{\mathcal{X}}(j)$. $M$ likewise generates a list of $\hat{d}_{\mathcal{Y}}(j)$. $M$ outputs $\hat{d}_{\mathcal{X}}(m_{\mathcal{X}})$ and $\hat{d}_{\mathcal{X}}(m_{\mathcal{Y}})$.

    Note that $d(X, r)$ appears on the list of $\hat{d}_{\mathcal{X}}(i)$, since $d(X, r)\in B_{2^{-r}}(X)\subseteq B_{2^{-(r-1)}}(\bar{X})$; the same is true for $d(Y, s)$. Hence, for the right choice of $m_\mathcal{X}$ and $m_\mathcal{Y}$, $M$ outputs $d(X, r)$ and $d(Y, s)$. Furthermore, by Lemma \ref{lem:uniformityForD},  $m_\mathcal{X}\leq C_{\mathcal{X}, 1}$ and $m_\mathcal{Y}\leq C_{\mathcal{Y}, 1}$. Combining the above information gives the bound
    \begin{align*}
        K^A(d(X, r), d(Y, s))&\leq K^A(\bar{X}, \bar{Y}) + K^A(r) + K^A(s) + K^A(\sigma_{\mathcal{X}}) + K^A(\sigma_{\mathcal{Y}})\\ &\qquad + K^A(m_{\mathcal{X}}) + K^A(m_{\mathcal{Y}}) + c_M\\
        &\leq K^A(\bar{X}, \bar{Y}) + O(\log(r+s)) + K^A(\sigma_{\mathcal{X}}) + K^A(\sigma_{\mathcal{Y}})\\
        &\qquad + O\log(C_{\mathcal{X}, 1}) + O\log(C_{\mathcal{Y}, 1})\\
        &\leq  K^A(\bar{X}, \bar{Y})  + O_{\mathcal{X}, \mathcal{Y}}(\log(r+s)).
    \end{align*}
    This completes our proof of \textit{(i)}. 
    
    \noindent \textit{Proof of (ii):} \textit{(ii)} is easily seen to be a consequence of \textit{(i)}; simply apply \textit{(i)} with $X$ as, say, a rational element of $\mathcal{X}$. It remains only to show that
    $K^A_{r, s}(X\mid Y) = K^A(d(X, r)\mid d(Y, s))\pm O(\log(r + s))$. 

    \noindent \textit{Proof of (iii):} We begin with the direction $K^A_{r, s}(X\mid Y) \leq K^A(d(X, r)\mid d(Y, s))\pm O(\log(r + s))$. Let $P\in B_{2^{-s}}(Y) \cap \mathbb{Q}_\mathcal{Y}$ be such that
    \begin{equation*}
    K^A_{r, s}(X\mid Y) = \min_{Q\in B_{2^{-r}}(X)\cap\mathbb{Q}_{\mathcal{X}}} K^A(Q \mid P).
\end{equation*}
At the cost of a small error term, we can move the (finite) amount of information in $P$ into the oracle and apply \textit{(ii)} relative to $(A, P)$, which gives
 \begin{align*}
    \min_{Q\in B_{2^{-r}}(X)\cap\mathbb{Q}_{\mathcal{X}}} K^A(Q \mid P) &= \min_{Q\in B_{2^{-r}}(X)\cap\mathbb{Q}_{\mathcal{X}}} K_r^{A, P}(Q) + O(\log r)\\
    &= K_r^{A, P}(X) + O(\log r)\\
    &\leq K^{A, P}(d(X, r)) + O(\log r)\\
     &= K^{A}(d(X, r)\mid P) + O(\log r ).
\end{align*}
The proof of the second half of \textit{(i)} shows $d(Y, s)$ is computable from any $2^{-s}$ approximation of $Y$ given little additional information, so $d(Y, s)$ cannot be much more helpful in the computation of $d(X, r)$ than $P$, and in particular we have 
\begin{equation*}
K^{A}(d(X, r)\vert P)\leq K^{A}(d(X, r)\vert d(Y, s)) + O(\log (r+s)),
\end{equation*}
which when combined with the above chain of inequalities completes the proof of the first direction of \textit{(iii)}. 

Moving onto the second direction, for any $P\in B_{2^{-s}}(Y) \cap \mathbb{Q}_\mathcal{Y}$, 

\begin{equation*}
K^A_{r, s}(X\mid Y) \geq \min_{Q\in B_{2^{-r}}(X)\cap\mathbb{Q}_{\mathcal{X}}} K^A(Q\mid P) - O(\log(r + s)).
\end{equation*}
So, in particular, 
\begin{align*}
K^A_{r, s}(X\mid Y) &\geq \min_{Q\in B_{2^{-r}}(X)\cap\mathbb{Q}_{\mathcal{X}}} K^A(Q\mid d(Y, s)) - O(\log(r + s))\\
&= \min_{Q\in B_{2^{-r}}(X)\cap\mathbb{Q}_{\mathcal{X}}} K^{A, d(Y, s)} (Q) - O(\log(r + s))\\
&=  K_r^{A, d(Y, s)} (X) - O(\log(r + s))\\
&\geq  K^{A, d(Y, s)} (d(X, r)) - O(\log(r + s))\\
&=  K^{A} (d(X, r)\mid d(Y, s)) - O(\log(r + s)),
\end{align*}
and the proof is complete. 
    \qed
   \end{proofof}
 
To end this subsection, we note one special case of the symmetry of information which expresses the information content of an element of the affine Grassmannian. 

\begin{cor}\label{cor:specialSoI}
For all $A\subseteq\mathbb{N}$ and $P= V + t$
    \begin{equation*}
    K^A_r(P) = K^A_r(V) + K^A_r(t\mid V) + O(\log r) = K^A_r(t) + K^A_r(V\mid t) + O(\log r).
    \end{equation*}
\end{cor}

\subsection{Geometric lemmas}

The main geometric lemma we will need shows how well a $k$-plane is determined by (approximations of) points on it. 

\begin{lem}\label{lem:pointsDeterminePlanes}
Let $P = V+t\in \mathcal{A}(n, k)$, and let $p_0, ...,  p_{k}$ be a set of points in $P$. Suppose $p_0, ...,  p_{k}$ is such that $(p_1 - p_{0}), ...,  (p_{k}-p_{0})$ spans $V\in \mathcal{G}(n, k)$. Then for every oracle $A\subseteq\mathbb{N}$ and $r\in \mathbb{N}$ sufficiently large (depending on $p_0, ...,  p_k$),
\begin{equation*}
    K^A_{r}(P, p_0, ...,  p_k)\leq K^A_r(p_0, ..., p_k) + O_{n, k, \sigma, \vert p_0\vert}(\log r),
\end{equation*}
where $\sigma$ is the smallest nonzero singular value of the matrix with $(y_1-y_{k+1}), ...,  (y_{k}-y_{k+1})$ as its columns. 
\end{lem}

\begin{remark} In our applications, it will be necessary to upper bound $K^A_{r}(P, p_0, ...,  p_k)$, which is why we have stated the lemma in this way. However, clearly 
\begin{equation*}
    K^A_{r}(P)\leq K^A_{r}(P, p_0, ...,  p_k)+O(\log r),
\end{equation*}
so Lemma \ref{lem:pointsDeterminePlanes} also establishes the more natural statement that 
\begin{equation*}
K^A_{r}(P)\leq K^A_r(p_0, ..., p_k) + O_{n, k, \sigma, \vert p_0\vert}(\log r).
\end{equation*}
\end{remark}

\begin{proof}
    It suffices to show that a precision $r$ approximation of the points $p_0, ..., p_k$ (which we denote $\bar{p}_0, ..., \bar{p}_k$) and a small amount of additional information is enough to compute a nearly precision $r$ approximation of $P$. We define a Turing machine $M$ that performs this computation. 

    $M$ takes as its input $(\bar{p}_0, ..., \bar{p}_k)$, $r$, $n$, and $k$. First, $M$ computes $\bar{p}_1 - \bar{p}_0, ..., \bar{p}_k - p_0$. Then, $M$ determines the rational projection matrix $\bar{V} = A(r)(A(r)^{\text{T}} A(r))^{-1}A(r)^{\text{T}}$, where $A(r)$'s columns are the vectors $\bar{p}_i - \bar{p}_0$. Note that the above inverse exists (provided that $r$ is sufficiently large) since the vectors $p_i - p_0$ are linearly independent. It follows immediately by Lemma \ref{lem:singularValue} that 
    \begin{equation}\label{eq:firstHalfGeo}
        \rho(V, \bar{V})\leq  C_{n, k} n \sigma^{-1} 2^{-(r-1)}.
    \end{equation}
    It remains to approximate $t$, and we can employ $\bar{V}$ to accomplish this. Specifically, $M$ calculates $\bar{t}^\prime =\bar{p}_0 - \bar{V}\bar{p}_0$. Note that by the triangle inequality, the equivalence of the metrics $m$ and $\rho$, and \eqref{eq:firstHalfGeo},
    \begin{align*}
       \vert t - \bar{t}^\prime\vert &\leq\vert p_0 - \bar{p}_0\vert  + \vert V p_0 - \bar{V}\bar{p}_0\vert \\
       &\leq \vert p_0 - \bar{p}_0 \vert + \vert V p_0  - V \bar{p}_0\vert + \vert V \bar{p}_0 - \bar{V} \bar{p}_0\vert\\
       &\leq 2^{-r} + 2^{-r} + 2 \vert p_0\vert  C_{n, k} \rho(V, \bar{V})\\
       &\leq (1 + 2 \vert p_0\vert  (C_{n, k})^2 n \sigma^{-1}) 2^{-(r-1)}.
    \end{align*}
We are not quite done, since $\bar{t}^\prime\in \mathbb{Q}^n$, but in our formalization, we need the \emph{coordinates} of $\bar{t}^\prime$ in the basis we have associated to $\bar{V}^\perp$. However, this is no major obstruction, since $M$ can compute the basis for $\bar{V}^\perp$ and sequentially test the elements of an enumeration of $\mathbb{Q}^{n-k}$ until it finds coordinates that give a vector $\bar{t}$ within $2^{-r}$ of $\bar{t}^\prime$. Finally, $M$ outputs $\bar{P} = (\bar{V}, \bar{t})$. We have that 

\begin{align*}
    \rho(P, \bar{P)}&= \rho(V, \bar{V}) + \vert t - \bar{t}\vert \\
    &\leq \rho(V, \bar{V})+ \vert \bar{t} - \bar{t}^\prime \vert + \vert t - \bar{t}^\prime\vert   \\
    &\leq \left(C_{n, k} n \sigma^{-1} + 2 +  2 \vert p_0\vert  (C_{n, k})^2 n \sigma^{-1}\right) 2^{-(r-1)}\\
    &:= C_{n, k, \sigma, \vert p_0\vert} 2^{-r}.
\end{align*}
This inequality and the construction of $M$ imply
\begin{equation}\label{eq:endOfFirstGeo}
    K^A_{r - \lceil \log C_{n, k, \sigma, \vert p_0\vert}\rceil}(P, p_0, ...,  p_k)\leq K^A_r(p_0, ..., p_k) + K^A(r) + K^A(n) + K^A(k) + c_M.
\end{equation}
To finish the proof, we need only apply \eqref{eq:endOfFirstGeo} at a higher precision, Proposition \ref{prop:caseLutz}, and some basic bounds on the smaller terms. In particular, letting $r^\prime = r + \lceil \log C_{n, k, \sigma, \vert p_0\vert}\rceil$, we have
\begin{align*}
    K^A_r(P, p_0, ...,  p_k)&\leq K^A_{r^\prime}(p_0, ..., p_k) + K^A(r) + K^A(n) + K^A(k) + c_M\\
    &\leq K^A_{r}(p_0, ..., p_k) + K^A_{r^\prime, r}(p_0, ..., p_k) + K^A(r) + K^A(n) + K^A(k) + c_M\\
    &\leq K^A_{r}(p_0, ..., p_k) + O_{n, k, \sigma, \vert p_0\vert}(\log r).
\end{align*}
\end{proof}
This lemma has as a consequence that an approximation of a $k$-plane allows us to determine an approximation of its orthogonal complement. More precisely
\begin{cor}\label{cor:planeDeterminesComplement}
    Let $V\in\mathcal{G}(n, k)$, a precision $r$, and $A\subseteq\mathbb{N}$ be given. 
\begin{equation*}
    K_r^A(V^\perp)\leq K_r^A(V) + O_{n, k}(\log r).
\end{equation*}
\end{cor}
The idea is simply that we can use $V$ to find vectors in $V^\perp$, and Lemma \ref{lem:pointsDeterminePlanes} indicates these vectors are enough to determine $V^\perp$. Formally,
\begin{proof}
    Let $n, k$ be given. Let $\sigma(v_1, ..., v_m)$ denote the smallest singular value of the matrix with $v_1, ..., v_m$ as columns and let $e_i$ denote the $i$th standard basis vector. It is straightforward to observe (say, by using the compactness of $\mathcal{G}(n, k)$) that there is some constant $\hat{C}_{n, k}$ such that,
    \begin{equation}\label{eq:complementSV}
    \inf_{V\in\mathcal{G}(n, k)}\max_{1\leq i_1< ...<i_{n-k}\leq n}\sigma(e_{i_1} - V e_{i_1}, ..., e_{i_1} - V e_{i_{n-k}})\geq \hat{C}_{n, k}>0.
    \end{equation}
    Let $\bar{V}$ be a precision $r$ approximation of $V$. There is a Turing machine $M$ that does the following. Given $n$, $k$, $\bar{V}$, and a string $\tau$ encoding which $n-k$-tuple achieves the maximum for $V$ in \eqref{eq:complementSV}, $M$ outputs the vectors $e_{i_1} - \bar{V} e_{i_1}, ..., e_{i_{n-k}} - \bar{V} e_{i_{n-k}}$. By the definition of the metric on $V$, 
    \begin{equation*}
    \vert (e_{i_1} - V e_{i_1}) - (e_{i_1} - \bar{V} e_{i_1})\vert \leq 2^{-r}.
    \end{equation*}
    Hence, 
    \begin{align*}
        K_r^A(0, e_{i_1} - V e_{i_1}, ..., e_{i_{n-k}} - V e_{i_{n-k}})&\leq K_r^A(V) + K_r^A(\tau) + O_{n, k}(\log r)\\
        &= K_r^A(V) + O_{n, k}(\log r),
    \end{align*}
    where the second line holds because the number of possible values of $\tau$ depends only on $n$ and $k$. The conclusion follows from applying Lemma \ref{lem:pointsDeterminePlanes} to the collection of vectors $0, e_{i_1} - V e_{i_1}, ..., e_{i_{n-k}} - V e_{i_{n-k}}$, which span $V^\perp$. Note that after this application, the logarithmic error term will still only depends on $n$ and $k$ because of the uniform bound of \eqref{eq:complementSV} and the fact that these vectors have norm $\leq 1$.
\end{proof}

We will also have need of the following simple fact, that points on a $k$-plane cannot have complexity too high given an approximation of that $k$-plane. 

\begin{lem}\label{lem:pointsOnAPlane}
    Suppose $P=V+t\in \mathcal{A}(n, k)$, and let $x\in P$. Then
\begin{equation*}
       K_r^A(x\mid P)\leq k r + O(\log r). 
\end{equation*}
\end{lem}

\begin{proof}
    Let $e$ be one of the standard coordinate $k$-planes such that the projection of $e$ onto $V$ has dimension $k$. Then if $x\in P$, there is some $x^\prime\in e$ such that $V x^\prime+t = x$. Let $r$ be sufficiently large and let $q$ be a precision $2^{-r}$ approximation of $x^\prime$ and $\bar{P} = \bar{V} + \bar{t}$ be a precision $r$ approximation of $P$. Using the fact that projections are $1$-Lipshitz and the definition of the metric on $\mathcal{A}(n, k)$, we have
    \begin{align*}
        \vert x - (\bar{V} q + \bar{t}) \vert &= \vert Vx^\prime+t - (\bar{V} q + \bar{t}) \vert\\
        &\leq \vert V x^\prime - V q\vert + \vert V q + \bar{V} q\vert + \vert t - \bar{t}\vert \\
        &\leq 2^{-r} + 2^{-r}\vert q \vert  + 2^{-r}\\
        &\leq 2^{-(r-2)}(1 + \vert x \vert).
    \end{align*}
    Clearly, there is a Turing machine that can perform the calculation $(\bar{V} q + \bar{t})$ given these approximations. Since precision $r$ approximations of $x^\prime$ and $P$ are sufficient to compute a precision $r^\prime = r-2-\log(1 + \vert x \vert)$ approximation of $x$, in conjunction with Proposition \ref{prop:caseLutz} we have 
\begin{align*}
   K_r^A(x\mid P)&\leq K^A_{r^\prime, r}(x\mid P) + O(\log r)\\
    &\leq K^A_r(x^\prime) + O(\log r).
\end{align*}
Since $x^\prime\in e$, all but $k$ of its coordinates are $0$, hence 
\begin{equation*}
    K_r^A(x^\prime)\leq k r + O(\log r),
\end{equation*}
and the conclusion follows. 
\end{proof}

Effective dimension has been used in the context of lines in higher dimensions; for instance, N. Lutz and Stull proved results for projections onto lines in $\mathbb{R}^n$ \cite{LutStu18Projections}. To end this section, we prove that the notion of complexity on the Grassmannian that we have introduced agrees with more the more intuitive notion for lines and hyperplanes. In particular, for lines, we have the following.
\begin{prop}\label{prop:lineEquiv}
    Let $\ell=V+t\in \mathcal{A}(n, 1)$ be the line given by the equation $\ell(x) = (x, a_1 x + b_1, ...,  a_{n-1} x + b_{n-1})$. Then for every $A\subseteq\mathbb{N}$ and precision $r$, 
    \begin{equation*}
        K^A_r(\ell) = K^A_r(a_1, b_1, ..., a_{n-1}, b_{n-1}) \pm O(\log r).
    \end{equation*}
\end{prop}
\begin{proof}

    First, note that the direction $K^A_r(\ell)\leq K^A_r(a_1, b_1, ..., a_{n-1}, b_{n-1}) + O(\log r)$ follows from Lemma \ref{lem:pointsDeterminePlanes} applied to the points $p_0 = (0, b_1, ..., b_{n-1})$ and $p_1 = (1, a_1, ..., a_{n-1})$, along with the observation that
    \begin{equation*}
        K_r^A(0, b_1, ...,  b_n, 1, a_1 + b_1, ..., a_{n-1} + b_{n-1})\leq K_r^A(a_1, b_1, ..., a_{n-1}, b_{n-1}) + O(\log r).
    \end{equation*}

    We now prove the other direction. Assuming that $p_1$ and $p_2$ are on $\ell$, it is straightforward to observe that
    \begin{equation}\label{eq:lineEquivDef1}
    K_{r- \log \vert p_1 - p_2\vert}^A(a_1, b_1, ..., a_{n-1}, b_{n-1})\leq K_r^A(p_1, p_2) + O(\log r).
    \end{equation}
    Briefly, this inequality holds because there is a Turing machine that uses precision $r$ approximations of $p_1$ and $p_2$ to determine an approximation of the line $\ell(s) = p_1 + s (p_2 - p_1)$ and then solves for the values of $s$ that make the first coordinate $0$ and $1$. So, it suffices to show that there exist $p_1, p_2\in\ell$ with approximations that are easily computable from approximations of $\ell$. Let $i$ maximize the norm of the projections of the standard basis vectors $e_j$ onto $V$. Let $p_1 = V e_i +t$ and $p_2 = -V e_i + t$ and note that $\vert p_1 - p_2\vert\geq \frac{1}{2\sqrt{n}}$.
    
    Given a precision $r$ approximation $\bar{\ell} = \bar{V} + \bar{t}$ of $\ell$ along with $i, r$, and $n$, there exists a Turing machine $M$ that calculates $\bar{p}_1 = \bar{V} e_i + \bar{t}^\prime$ and $\bar{p}_2= - \bar{V} e_i + \bar{t}^\prime$, where $\bar{t}^\prime$ is some $2^{-r}$ approximation of $\bar{t}$. Then, by the definition of $\rho$, 
    \begin{align*}
        \vert p_1 - \bar{p}_1\vert&\leq \vert V e_i - \bar{V} e_i\vert + 2^{-(r-1)}\\
        &\leq \rho(V, \bar{V}) + 2^{-(r-1)}\\
        &\leq 2^{-(r-2)},
    \end{align*}
    and likewise for $p_2$. Hence,
    \begin{align*}
        K_{r}^A(p_1, p_2)&\leq K^A_{r-3}(p_1, p_2) + O(\log r)\\
        &\leq K_r^A(\ell) + K^A(r) + K^A(i) + K^A(n) + c_M + O(\log r)\\
        &\leq K_r^A(\ell) + O(\log r).
    \end{align*}
    Combining the above with \eqref{eq:lineEquivDef1} and recalling that the distance between $p_1$ and $p_2$ depended only on $n$ completes the proof. 
    \end{proof}
    
Now we state an analogous result for hyperplanes.
\begin{prop}\label{prop:hyperplaneEquiv}
       Let $P = V+t\in \mathcal{A}(n, n-1)$ be the hyperplane given by $(x_1, ..., x_{n-1}) \mapsto (x_1, ..., x_{n-1}, a_1 x_1 + ... + a_{n-1} x_{n-1} + b)$. Then for every $A\subseteq\mathbb{N}$ and precision $r$, 
    \begin{equation*}
        K^A_r(P) = K^A_r(a_1, ..., a_{n-1}, b) \pm O(\log r).
        \end{equation*}
\end{prop}
\begin{proof}
    As in the proof of Proposition \ref{prop:lineEquiv}, the direction $K^A_r(P) \leq K^A_r(a_1,  ..., a_{n-1}, b) + O(\log r)$ follows from Lemma \ref{lem:pointsDeterminePlanes}; in this case we apply the lemma to the points 
    \begin{equation*}
    \{p_0, p_1, ..., p_{n-1}\} = \{(0, ..., 0, b), (1, 0, ..., 0, a_1 + b), ...,(0, ..., 0, 1, a_{n-1} + b)\}. 
    \end{equation*}
    
    The proof of the other direction is also very similar to that of Proposition \ref{prop:lineEquiv}. Provided that the vectors $p_0 - p_1, ..., p_0 - p_{n-1}$ span $V$, we have 
    \begin{equation}\label{eq:hyperplaneEquiv1}
        K_{r- \log  n\sigma^{-1}}^A(a_1, ..., a_{n-1}, b) \leq K^A_r(p_0, ..., p_{n-1}) + O(\log r),
    \end{equation}
    where $\sigma$ is the smallest singular value of the column matrix formed by the vectors $p_i - p_0$. The idea is essentially the same as in Proposition \ref{prop:lineEquiv}: a Turing machine can use approximations of these vectors to give a parameterization of the hyperplane and then generate the points in the first direction of the proof, from which it can read off $a_1, ..., a_{n-1}, b$. In this case, we can apply Lemma \ref{lem:singularValue} to bound the error in the machine's approximation of the hyperplane. 

    It remains to generate $p_0, ..., p_{n-1}$ from $P$. Let $i$ maximize the smallest singular value of the column matrix formed by $V e_1, ..., V e_{j-1}, V e_{j+1}, V e_n$ and note that we can lower bound this maximal smallest singular value in terms of $n$ uniformly for all hyperplanes $P$. There is a Turing machine $M$ that takes $\bar{P} = \bar{V} + \bar{t}, r, i$, and $n$ as inputs and calculates the vectors $\bar{p}_0 = \bar{t}$ and $\bar{p}_j = \bar{V} e_{j} + \bar{t}$; these are approximations of the vectors $p_i$. The same error estimate as in the proof of Proposition \ref{prop:lineEquiv} holds, hence we have that for each $i$,
    \begin{equation*}
        \vert p_i - \bar{p_i}\vert \leq 2^{-(r-2)}, 
    \end{equation*}
    and thus, 
    \begin{equation*}
        K^A_{r-2-\log n}(p_0, ..., p_{n-1})\leq K_r^A(P) + K^A(r) + K^A(n) + K^A(i) + c_M.
    \end{equation*}
Combining this inequality with \eqref{eq:hyperplaneEquiv1} and then applying Proposition \ref{prop:caseLutz} (and basic bounds on the size of the smaller terms) completes the proof.  
\end{proof}

\section{The packing dimension of collections of $k$-planes}

In this section, we prove Theorem \ref{thm:planeUnion} and Theorem \ref{thm:purePackingBound}. As indicated in the introduction, the bound in Theorem \ref{thm:purePackingBound} is weaker because we cannot in general expect the scales at which the set of $k$-planes is large to coincide with the scales at which the subsets of the $k$-planes are large. Consequently, however, its proof is similar to but slightly simpler than the proof of Theorem \ref{thm:planeUnion}. Hence, we begin this section with Theorem \ref{thm:purePackingBound}. 

\begin{T2}
    Suppose $\mathcal{P}\subseteq\mathcal{A}(n, k)$ has packing dimension $t>0$. Let $F=\bigcup_{P\in\mathcal{P}}P^\prime$, where each $P^\prime\subseteq P$ has packing dimension at least $s>0$. Then,
    \begin{equation*}
        \dim_P(F)\geq \max\{s, \frac{t- (k-\lceil s\rceil)(n-k)}{\lceil s\rceil +1}\}.
    \end{equation*}
\end{T2}

Before beginning the proof, we briefly outline it. The optimal situation for us is when $s>k-1$. This means that every $P\in \mathcal{P}$ is spanned by a set of points in $F$. Hence, (relative to an appropriate oracle) we can bound the complexity of these points in terms of $\dim_P(F)$. Since these points determine $P$, Lemma \ref{lem:pointsDeterminePlanes} indicates that $P$'s complexity is no higher than the total complexity of the points. We get the bound
\begin{equation*}
\dim_P(F)\geq \frac{t}{k + 1}
\end{equation*}
by picking $P$ so that it has high complexity relative to a packing oracle for $\mathcal{P}$.

When we are not in this optimal situation, we are potentially forced to use points outside of $F$ to determine $P$, which gives us a worse bound. However, we can do two things in our selection of these extra points that mitigate the damage. First, we can pick them to be in the orthogonal complement of the space spanned by the points in $P^\prime$. Applying Lemma \ref{lem:pointsOnAPlane}, each additional point cannot add \emph{too much} complexity relative to what we already have to compute. Second, even with the above requirement, we still have some freedom in the selection of these extra points, enough that we can ensure each one contains appreciable information \emph{unrelated} to $P$. Since our eventual list of points determines more than just $P$, we get a better lower bound on the complexity of these points than we otherwise would. We now proceed with the formal proof. 

\begin{proof}
    $F$ contains a packing dimension $s$ subset, so $\dim_P(F)\geq s$. To bound $F$ it in terms of $t$, let $A$ be a packing oracle for $F$ and $\mathcal{P}$ and let $\ve>0$ be given. Let $P\in \mathcal{P}$ have effective packing dimension relative to $A$ of at least $t-\ve$. We begin by finding a collection of vectors $p_0, ..., p_k$ that determine $P$ in the sense of Lemma \ref{lem:pointsDeterminePlanes}. We want as many of these vectors as possible to be from $P^\prime$. 
    
    Since $P^\prime$ has packing dimension at least $s$, we can pick $p_0, ..., p_{\lceil s \rceil+1}\in P$ such that $(p_1 -p_0), ..., (p_{\lceil s \rceil} - p_0)$ are linearly independent; let $P_s$ denote the span of these vectors. We will choose $p_{\lceil s \rceil+1}, ..., p_k \in (P-p_0)\cap (P_s)^\perp$, also linearly independent. Additionally, since $(P-p_0)\cap (P_s)^\perp$ is a linear subspace with dimension $k - \lceil s \rceil$, there is a standard coordinate $(k - \lceil s \rceil)$-plane $e$ such that $\pi_e:(P-p_0)\cap (P_s)^\perp\to e$ is bi-Lipschitz. In particular, the image of the projection has \emph{Hausdorff} dimension $k - \lceil s \rceil$. By the point-to-set principle, then, we can require that for each $i$ and for \emph{all} sufficiently large $r$, 
    \begin{equation}\label{eq:purePackingBound1}
    K_r^{A, p_1, ..., p_{i-1}}(\pi_e p_i)\geq (k - \lceil s \rceil - \ve)r.
    \end{equation}

    We will make use of this fact later. For now, note that the vectors $p_0, ...p_{\lceil s \rceil}, p_{\lceil s \rceil+1} + p_0, ..., p_k+p_0$ are as in the statement of Lemma \ref{lem:pointsDeterminePlanes}, and clearly 
    \begin{equation*}
K_r^A(p_0, ...p_{\lceil s \rceil+1}, p_{\lceil s \rceil+1} + p_0, ..., p_k)\leq K_r^A(p_0, ..., p_k)+O(\log r),
    \end{equation*}
    so
    \begin{align}
       K_r^A(P, p_{\lceil s \rceil +1}, ..., p_k)&\leq K_r^A(P, p_0, ..., p_k)+O(\log r)\nonumber\\
       &\leq K_r^A(p_0, ..., p_k) + O(\log r).\label{eq:purePackingPtop}
    \end{align}
Throughout this proof (and the next) we will not explicitly note the dependence of the logarithmic errors on $n, k$ or the properties of the collection of points, since we can assume $r$ is sufficiently large depending on these fixed terms. Now, we apply the symmetry of information, which will let us handle the points from $P^\prime$ differently than the remaining ones we choose.
    \begin{align}
            K_r^A(p_0, ..., p_k) &= K^A_r(p_0,  ..., p_{\lceil s \rceil}) + K^A_r(p_{\lceil s\rceil +1},  ..., p_k\mid p_0,  ..., p_{\lceil s \rceil})  + O(\log r)\nonumber\\
            &\leq\sum_{j=0}^{\lceil s\rceil} K_r^A(p_j) +   \sum_{j=\lceil s\rceil +1}^{k} K_r^A(p_j\mid p_0,  ..., p_{\lceil s \rceil})) + O(\log r).   \label{eq:purePackingSums}
    \end{align}
    Since $A$ is a packing oracle for $F$ and $p_0, ..., p_{\lceil s\rceil}\in P^\prime \subset F$, we can bound the $\lceil s\rceil + 1$ terms in the first sum by $(\dim_P(F) + \ve) r$ (provided that $r$ is sufficiently large). Fix $j\in\{\lceil s\rceil +1, ..., k\}$.  For the remaining terms, we can obtain a better bound than the trivial $\approx nr$ thanks to some of the geometric lemmas. More precisely, by Lemma \ref{lem:pointsDeterminePlanes}, a roughly precision $r$ approximation of $P_s$ is computable from $p_0, ..., p_{\lceil s\rceil}$. Hence,
    \begin{equation*}
        K_r^A(p_j\mid p_0,  ..., p_{\lceil s \rceil}) \leq K_r^A(p_j\mid P_s) + O(\log r).
    \end{equation*}
    Corollary \ref{cor:planeDeterminesComplement} indicates that a roughly precision $r$ approximation of $(P_s)^\perp$ is computable from $P_s$, so
    \begin{equation*}
        K_r^A(p_j\mid p_0,  ..., p_{\lceil s \rceil}) \leq K_r^A(p_j\mid (P_s)^\perp) + O(\log r).
    \end{equation*}
    Now, we can use the fact that $p_j$ was chosen to be in $(P_s)^\perp$, a subspace of dimension $n - \lceil s \rceil$; applying Lemma \ref{lem:pointsOnAPlane} immediately gives that
      \begin{equation}\label{eq:purePackingSecondUB}
        K_r^A(p_j\mid p_0,  ..., p_{\lceil s \rceil}) \leq K_r^A(p_j\mid (P_s)^\perp) + O(\log r)\leq (n - \lceil s\rceil ) r + O(\log r). 
    \end{equation}
    Combining \eqref{eq:purePackingSums} with \eqref{eq:purePackingSecondUB} and the prior observation that $A$ is a packing oracle for $F$ implies
    \begin{equation*}
        K_r^A(p_0, ..., p_k)\leq (\lceil s \rceil + 1) (\dim_P(F) + \ve) + (k - \lceil s \rceil)(n - \lceil s\rceil) r + O(\log r).
    \end{equation*}
    Together with \eqref{eq:purePackingPtop}, if $r$ is sufficiently large, we have 
    \begin{equation}\label{eq:purePackingMainUB}
        K_r^A(P, p_{\lceil s \rceil +1}, ..., p_k) \leq (\lceil s \rceil + 1) \dim_P(F)r  + (k - \lceil s \rceil)(n - \lceil s\rceil) r + n\ve r.
    \end{equation}
    Now, we turn towards lower bounding $K_r^A(P, p_{\lceil s \rceil +1}, ..., p_k)$. Clearly, $K_r^A(P)$ is a lower bound, but we can get an improved bound if we use the fact that $p_{\lceil s \rceil + 1}, ..., p_k$ were chosen to have (projections with) high complexity relative to $A, P$, and all the previous points. Since $e$ was one of the standard coordinate $k$-planes, the projection map $\pi_e$ is computable and thus
    \begin{equation*}
        K_r^A(P, \pi_e p_{\lceil s \rceil +1}, ..., \pi_e p_k)\leq  K_r^A(P, p_{\lceil s \rceil +1}, ..., p_k).
    \end{equation*}
    Our goal is to take apart the left hand side with repeated applications of the symmetry of information. In the following chain of inequalities, we also freely use fact that oracle access is no less helpful than precision $r$ access and the fact that $\pi_e$ is computable,
    \begin{align*}
        K_r^A(P, p_{\lceil s \rceil +1}, ..., p_k)&\geq K^A_r(P) + K^A_r(\pi_ep_{\lceil s\rceil + 1}, ..., \pi_{e}p_k\mid P) + O(\log r)\\
        &\geq K^A_r(P) + K^{A, P}_r(\pi_ep_{\lceil s\rceil + 1}, ..., \pi_{e}p_k) + O(\log r)\\
        &= K^A_r(P) + K^{A, P}_r(\pi_ep_{\lceil s\rceil + 1}) + K^{A, P}_r(\pi_ep_{\lceil s\rceil + 2} ..., \pi_{e}p_k\mid \pi_ep_{\lceil s\rceil + 1}) + O(\log r)\\
        &\geq K^A_r(P) + K^{A, P}_r(\pi_ep_{\lceil s\rceil + 1}) + K^{A, P, \pi_ep_{\lceil s\rceil + 1}}_r(\pi_ep_{\lceil s\rceil + 2} ..., \pi_{e}p_k) + O(\log r)\\
        &\geq K^A_r(P) + K^{A, P}_r(\pi_ep_{\lceil s\rceil + 1}) + K^{A, P, p_{\lceil s\rceil + 1}}_r(\pi_ep_{\lceil s\rceil + 2} ..., \pi_{e}p_k) + O(\log r)\\
        &\qquad ... \qquad ... \qquad ... \qquad ... \qquad ... \qquad ...\\
        &\geq K^A_r(P) + K^{A, P}(\pi_ep_{\lceil s\rceil + 1}) + ... + K^{A, P, p_{\lceil s\rceil + 1}, ..., p_{k - 1}}(\pi_{e}p_k) + O(\log r).
    \end{align*}
    Only the first line requires Proposition \ref{prop:symmetry}; the rest of the applications of the symmetry of information are to points in Euclidean space. We are now able to apply \eqref{eq:purePackingBound1} to each projection, giving
    \begin{equation*}
 K_r^A(P, p_{\lceil s \rceil +1}, ..., p_k)\geq K^A_r(P) + (k - \lceil s \rceil) (k - \lceil s\rceil) r - n \ve r.
 \end{equation*}
 Combining this bound with \eqref{eq:purePackingMainUB} we have that
 \begin{equation*}
     K^A_r(P) + (k - \lceil s \rceil) (k - \lceil s\rceil) r \leq (\lceil s \rceil + 1) \dim_P(F) r + (k - \lceil s \rceil)(n - \lceil s\rceil) r + 2 n\ve r.
 \end{equation*}
 Dividing by $r$, taking the limit superior, and recalling that $P$ had effective packing dimension close to $t$ (relative to $A$), we obtain
  \begin{equation*} 
     t + (k - \lceil s \rceil) (k - \lceil s\rceil) \leq (\lceil s \rceil + 1) \dim_P(F) + (k - \lceil s \rceil)(n - \lceil s\rceil) + 3 n\ve. 
 \end{equation*}
 Rearranging and recalling that $\ve>0$ was arbitrary completes the proof. 
    
\end{proof}
Now we turn our attention to Theorem \ref{thm:planeUnion}, which we restate here for convenience. 

\begin{T1}
    Suppose $\mathcal{P}\subseteq\mathcal{A}(n, k)$ has packing dimension $t>0$. Let $F=\bigcup_{P\in\mathcal{P}}P^\prime$, where each $P^\prime\subseteq P$ has Hausdorff dimension at least $s>0$. Then,
    \begin{equation*}
        \dim_P(F)\geq s + \frac{t - (k - \lceil s\rceil)(n-k)}{\lceil s\rceil +1}.
    \end{equation*}
    In particular, the union of a packing dimension $t$ set of $k$-planes has packing dimension at least $\frac{t}{k+1} + k$.
\end{T1}

In the proof of Theorem \ref{thm:purePackingBound}, we placed no requirements on the $\lceil s \rceil +1$ points we picked from $P^\prime$ other than that they were linearly independent. By contrast, the remaining $k-\lceil s \rceil$ points were chosen more carefully, which was fruitful because the Hausdorff dimension of the set they came from was large, meaning we could ensure their complexity was high relative to whatever we wanted at every sufficiently large precision. Now, the Hausdorff dimension of $P^\prime$ is large, so we can choose the points from $P^\prime$ so that they \emph{also} contain information unrelated to $P$ at every sufficiently large precision. This is essentially the only difference in the argument \emph{mutatis mutandis}, but we still provide a proof of Theorem \ref{thm:planeUnion}.

\begin{proof}
    Let $A$ be a packing oracle for $F$ and $\mathcal{P}$. Let $\ve>0$ be given. By the point-to-set principle, we can pick $P\in\mathcal{P}$ with effective packing dimension relative to $A$ of at least $t-\ve$. Let $e$ denote one of the standard coordinate $k$-planes such that the projection $\pi_eP$ has dimension $k$. Then $\pi_e:P\to e$ is bi-Lipschitz, hence $\pi_eP^\prime$ has Hausdorff dimension $s$. Pick $p_0\in P^\prime$ such that \begin{equation*}
        \dim^{A, P}(\pi_e p_0)\geq s - \ve.
    \end{equation*}  
    Then, for each $i\in\{1, ..., \lceil s \rceil\}$, choose $p_i\in P^\prime$ such that
    \begin{itemize} 
        \item $\dim^{A, P, p_0, ..., p_{i-1}}(\pi_ep_i)\geq s-\frac{\ve}{2}$,
        \item $p_i - p_0$ is not in the span of $p_1 - p_0, ..., p_{i-1} - p_0$.
    \end{itemize}
    By observing that the set spanned by the (differences of the) vectors up to step $i$ has dimension $i\leq \lceil s\rceil -1<s$ and applying the point-to-set principle, it is clear that we can make such a choice. The second property guarantees that $(p_1 - p_{0}), ...,  (p_{\lceil s \rceil}-p_{0})$ spans some $\lceil s \rceil$-dimensional subspace of $P-p_{0}\in \mathcal{G}(n, k)$; let $P_s$ denote this subspace. In order to apply Lemma \ref{lem:pointsDeterminePlanes}, again we need to determine $k-\lceil s \rceil$ additional vectors that together with the previous vectors span $P-p_0$.

    We will choose them in $\hat{P}_s:=(P - p_0)\cap (P_s)^\perp)$ Let $e^\prime$ be one of the standard coordinate $(k-\lceil s\rceil)$-planes such that $\pi_{e^\prime}(\hat{P}_s)$ has full dimension. For each $i\in\{\lceil s\rceil+1, ..., k\}$, choose $p_i\in P^\prime$ such that
    \begin{itemize} 
        \item $\dim^{A, P, p_0, ..., p_{i-1}}(\pi_{e^\prime} p_i)=k - \lceil s \rceil$,
        \item $p_i$ is not in the span of $p_{\lceil s\rceil+1}, ..., p_{i-1}$.
    \end{itemize}
    It is clear that our collection of vectors $p_0, ..., p_k$ is sufficient to determine $P$, hence we may apply Lemma \ref{lem:pointsDeterminePlanes}. In combination with the symmetry of information, this gives
    
    \begin{align*}
        K_r^A(P, p_0, ..., p_k)&\leq K^A_r(p_0,  ..., p_k) + C_{n, k, \sigma} + O(\log r)\\
        &= K^A_r(p_0,  ..., p_{\lceil s \rceil}) + K^A_r(p_{\lceil s\rceil +1},  ..., p_k\mid p_0,  ..., p_{\lceil s \rceil})  + O(\log r).
    \end{align*}
    Upper bounding $K^A_r(p_0,  ..., p_{\lceil s \rceil})$ is again straightforward; computing these vectors together is no more difficult than computing them separately, and $A$ is a packing oracle for $F$ (which contains $p_0, ...,  p_{\lceil s\rceil}$), so for any $\ve>0$, if $r$ is sufficiently large, 
    \begin{equation}\label{eq:firstHalfUpperBound}
        K^A_r(p_0,  ..., p_{\lceil s \rceil})\leq (\lceil s\rceil + 1) \dim_P(F) r+ \frac{\ve}{2} r.
    \end{equation}
    Upper bounding the second term is similar, in that we clearly have
    \begin{equation*}
        K^A_r(p_{\lceil s\rceil +1},  ..., p_k\mid p_0,  ..., p_{\lceil s \rceil}) \leq \sum_{j=\lceil s\rceil + 1}^k K_r^A(p_j\mid p_0,  ..., p_{\lceil s \rceil}) + O(\log r).
    \end{equation*}
    Bounding each term in the sum again follows from an application of Lemma \ref{lem:pointsDeterminePlanes}, Corollary $\ref{cor:planeDeterminesComplement}$, and Lemma \ref{lem:pointsOnAPlane}, which together imply
\begin{equation}\label{eq:secondHalfUpperBound}
        K^A_r(p_j\mid p_0,  ..., p_{\lceil s \rceil})\leq (n-\lceil s \rceil) r + O(\log r).
    \end{equation}
    Combining \eqref{eq:firstHalfUpperBound} and \eqref{eq:secondHalfUpperBound}, we obtain the main upper bound: 
    \begin{equation}\label{eq:unionPlanesLemmaApp}
    K_r^A(P, p_0, ..., p_k)\leq  (\lceil s\rceil + 1)\dim_P(F) r + (k - \lceil s \rceil)(n - \lceil s\rceil) r + \ve r.
    \end{equation}
    
    Now, we transition to lower bounding the same quantity. Since $e$ and $e^\prime$ are standard coordinate planes, the orthogonal projection maps onto them are computable and thus we have
    \begin{equation}\label{eq:projectionBasicBound}
        K_r^A(P, \pi_ep_0, ..., \pi_{e^\prime}p_k) \leq K_r^A(P, p_0, ..., p_k) + O(\log r).
    \end{equation}
    Repeatedly applying the symmetry of information, the fact that oracle access is no less helpful than precision $r$ access, and the fact that $\pi_e$ and $\pi_{e^\prime}$ are computable, we have that
    \begin{align*}
        K_r^A(P, \pi_ep_0, ..., \pi_{e^\prime}p_k)&=K^A_r(P) + K^A_r(\pi_ep_0, ..., \pi_{e^\prime}p_k\mid P) + O(\log r)\\
        &\geq K^A_r(P) + K^{A, P}_r(\pi_ep_0, ..., \pi_{e^\prime}p_k) + O(\log r)\\
        &= K^A_r(P) + K^{A, P}_r(\pi_ep_0) + K^{A, P}_r(\pi_ep_1 ..., \pi_{e^\prime}p_k\mid \pi_ep_0) + O(\log r)\\
        &\geq K^A_r(P) + K^{A, P}_r(\pi_ep_0) + K^{A, P, \pi_ep_0}_r(\pi_ep_1 ..., \pi_{e^\prime}p_k) + O(\log r)\\
        &\geq K^A_r(P) + K^{A, P}_r(\pi_ep_0) + K^{A, P, p_0}_r(\pi_ep_1 ..., \pi_{e^\prime}p_k) + O(\log r)\\
        &\qquad ... \qquad ... \qquad ... \qquad ... \qquad ... \qquad ...\\
        &\geq K^A_r(P) + K^{A, P}(\pi_ep_0) + ... + K^{A, P, p_0, ..., p_{k - 1}}(\pi_{e^\prime}p_k) + O(\log r).
    \end{align*}
    Using the construction of the $p_i$'s and the definition of effective Hausdorff dimension along with \eqref{eq:projectionBasicBound}, we have that for sufficiently large $r$, 
    \begin{equation}\label{eq:unionsOfPlanes1}
        K_r^A(P, p_0, ..., p_k)\geq K^A_r(P) + (\lceil s \rceil +1) (s - \ve) r + (k - \lceil s \rceil) (k - \lceil s\rceil-\ve) r.
    \end{equation}
    Combining \eqref{eq:unionPlanesLemmaApp} and \eqref{eq:unionsOfPlanes1}, we have
    \begin{equation*}
         K^A_r(P)\leq (\lceil s \rceil + 1)(\dim_P(F) - s) r+(k-\lceil s\rceil)(n-k)r + n\ve r.
    \end{equation*}
    Dividing by $r$ and taking the limit superior gives
   \begin{equation*}
         t-\ve\leq (\lceil s \rceil + 1)(\dim_P(F) - s)+(k-\lceil s\rceil)(n-k)+n\ve. 
    \end{equation*}
    Since $\ve>0$ was arbitrary, rearranging gives the desired conclusion.
\end{proof}

\begin{remark}
    We deliberately stated our theorems in a way that depended only on the dimension of the set of planes and the dimension of each subset. However, in both instances, we can obtain a better bound if each $P^\prime\subseteq P$ -- in addition to having Hausdorff or packing dimension at least $s$ -- obeys some suitable non-concentration condition. 
    \begin{itemize}
        \item In the proof of Theorem \ref{thm:purePackingBound}, we only needed our vectors $p_0, ..., p_{\lceil s\rceil}$ to be such that $p_1 - p_0, ..., p_{\lceil s\rceil}-p_0$ were linearly independent. Let $\alpha$ denote the minimum over $P\in\mathcal{P}$ of the dimension of the span (in the above sense) of $P^\prime$. Note that $\alpha\geq s$. The larger $\alpha$ is, the less subsets of $k$-planes can concentrate into lower-dimensional subspaces. Indeed, slightly modifying the proof actually gives
        \begin{equation*}
            \dim_P(F)\geq \max\{s, \frac{t- (k-\alpha)(n-k)}{\alpha +1}\}.
        \end{equation*}
        \item Because we asked for more from our choices of $p_i$ in the proof of Theorem \ref{thm:planeUnion}, the non-concentration condition is not quite as simple in this situation; $\alpha$ as defined above is not relevant. We need to be sure that each extra vector we add leaves a Hausdorff dimension $s$ set outside of the new span, so that we can choose the remaining vectors in $P^\prime$ to have high complexity relative to the appropriate oracles. More precisely, let
        \begin{equation*}
            \beta:=\min_{P\in \mathcal{P}}\max\{m\in\mathbb{N}:\forall p_0,...,  p_m\in P^\prime, \dim_H(P^\prime\setminus S(p_0, ... p_m))\geq s \},
        \end{equation*}
        where $S(p_0, ..., p_m) := p_0 + \text{span}(p_1-p+0, ..., p_m - p_0)$. Note that $\beta\geq s$. The larger $\beta$ is, the less concentrated the subsets of $k$-planes are into lower-dimensional subspaces. Then, following the proof with the necessary modifications yields
        \begin{equation*}
           \dim_P(F)\geq s+ \frac{t- (k-\beta)(n-k)}{\beta +1}.
        \end{equation*}
    \end{itemize}
\end{remark}

\section{Extensions of $k$-planes}

In this section, our main goal is to prove Theorem \ref{thm:basicExtension}, which we restate for convenience

\begin{T3}
Let $E\subseteq\mathbb{R}^n$ and let $F$ denote the union of $E$ with every $k$-plane that intersects $E$ in a set of positive measure. Then either $E=F$, or   

\begin{equation*}
    \dim_P(F)\leq 2 \dim_P(E) - k.
\end{equation*}

\end{T3}

\begin{proof}
We may assume $E$ contains at least on positive measure subset of a $k$-plane, or else clearly $E=F$. First, we reduce to the case that $E$ is the union of positive measure subsets of $k$-planes. Let $E^\prime$ be the maximal subset of $E$ of this form and let $F^\prime$ denote the extension (in the sense of the theorem) of $E^\prime$. Then $F = E \cup F^\prime$, so $\dim_P(F) = \max \{\dim_P(E), \dim_P(F^\prime)\}$. If the first term is larger, we are already done, so we assume the second is larger.

 Assume $E=E^\prime$ and let $\mathcal{I}$ index the $k$-planes intersecting $E$ in a set of positive measure. Applying the point-to-set principle,

\begin{equation*}
    \dim_P (E) = \min_{A \subseteq \mathbb{N}} \sup_{x \in E} \Dim^A(x) = \min_{A \subseteq \mathbb{N}} \sup_{I \in \mathcal{I}} \sup_{x\in E\cap I} \Dim^A(x),
\end{equation*}
and,
\begin{equation*}
    \dim_P (F) = \min_{A \subseteq \mathbb{N}} \sup_{x \in F} \Dim^A(x) = \min_{A \subseteq \mathbb{N}} \sup_{I \in \mathcal{I}} \sup_{x\in I} \Dim^A(x).
\end{equation*}
 So it suffices to show that for every $A\subseteq\mathbb{N}$, $I\in \mathcal{A}(n, k)$, and $S$ a positive measure subset of $I$, 
\begin{equation*}
    \sup_{x \in I} \Dim^A(x) \leq  2 \sup_{x \in S} \Dim^A(x) - k.
\end{equation*}
Using the properties of the limit superior this follows if for every $x\in I$ and $\ve>0$, there exist $y, z\in S$ such that for $r$ sufficiently large, 
\begin{equation}\label{eq:basicExtensions}
K^A_r(x) \leq K^A_r(y) + K^A_r(z) - k r + \ve r.
\end{equation}

 Assume without loss of generality that the projection of $S$ onto the first $k$-coordinates has positive measure. Choose $y$ so that its first $k$ coordinates are random relative to $x$ and $A$; call this vector $\hat{y}$. By Fubini's theorem, we may also assume the line containing $x$ and $y$ intersects $S$ in a set of positive measure. Hence, we may choose $z$ on this line such that one of its first $k$ coordinates is random relative to $x, y$, and $A$ ; call this coordinate $\hat{z}$.

Since $x, y$, and $z$ are collinear, there exists some $t\in \mathbb{R}$ such that $x = y + t(z - y)$. It is clear that a precision $r$ estimate of $y, z$, and $t$ is enough to compute a (nearly) precision $r$ estimate of $x$.\footnote{For instance, once can view this as an application of Lemma \ref{lem:pointsDeterminePlanes}.} Additionally, a Turing machine can read off coordinates of $y$ and $z$, so we have
\begin{align*}
    K^A_r(x, \hat{y}, \hat{z}) &\leq K^A_r(y, z, t) + O(\log r)\\
    &\leq K^A_r(y) + K^A_r(z) + K^A_r(t) + O(\log r)\\
    &\leq K^A_r(y) + K^A_r(z) + r + O(\log r).
\end{align*}
On the other hand, by our choices of $y$ and $z$, 
\begin{align*}
    K^A_r(x, \hat{y}, \hat{z}) &= K^A_r(x) +  K_r^A(\hat{y}, \hat{z})- O(\log r)\\
    &\geq K^A_r(x) +  K_r^{A, x}(\hat{y}) + K_r^{A, x, y}(\hat{z})- O(\log r)\\
    &= K^A_r(x) + kr + r - O(\log r).
\end{align*}
Combining these inequalities gives \eqref{eq:basicExtensions}. 
\end{proof}

    Having completed the proof, we note that the reason we assumed the intersections of $E$ with $k$-planes had positive measure was so we could guarantee that many of the radial slices (centered at $x$) were large. This was necessary so we could choose $z$ to have high complexity. However, in the case $k=1$, there is only one radial slice: the line itself. Thus we have the following improvement.  

\begin{P4}
Let $E\subseteq\mathbb{R}^n$ and let $F$ denote the union of $E$ with every line that intersects $E$ in a set of Hausdorff dimension 1. Then either $E=F$, or   
    \begin{equation*}
    \dim_P(F)\leq 2 \dim_P(E) - 1.
\end{equation*}
\end{P4}
\begin{proof}
    The first portion of the proof is identical to that of Theorem \ref{thm:basicExtension}, so it suffices to show 
    that for every $S\subseteq I\in\mathcal{A}(n, 1)$ of Hausdorff dimension 1, $x\in I$, and $\ve>0$, there exist $y, z\in S$ such that for $r$ sufficiently large, 
\begin{equation*}
K^A_r(x) \leq K^A_r(y) + K^A_r(z) - r + \ve r.
\end{equation*}
Assume without loss of generality that the projection of $S$ onto the first coordinate has Hausdorff dimension 1. By the point-to-set principle, we may choose $y$ and $z$ such that $\dim^{A, x}(\hat{y})\geq 1 - \frac{\ve}{4}$ and $\dim^{A, x, y}(\hat{z})\geq 1 - \frac{\ve}{4}$. Then for $r$ sufficiently large,
\begin{equation*}
    K_r^{A, x}(\hat{y})\geq (1 - \frac{\ve}{3}) r \qquad \text{ and } \qquad K_r^{A, x, y}(\hat{y})\geq (1 - \frac{\ve}{3}) r,
\end{equation*}
which implies 
\begin{align*}
    K^A_r(x, \hat{y}, \hat{z}) &\geq K^A_r(x) -  K_r^{A, x}(\hat{y}, \hat{z})+ O(\log r)\\
    &\geq K^A_r(x) +  K_r^{A, x}(\hat{y}) + K_r^{A, x, y}(\hat{z})- O(\log r)\\
    &\geq K^A_r(x) + 2r -\frac{2 \ve}{3}r - O(\log r).
\end{align*}
As in the case of Theorem \ref{thm:basicExtension}, we have that
\begin{equation*}
    K^A_r(x, \hat{y}, \hat{z}) \leq K^A_r(y) + K^A_r(z) + r + O(\log r).
\end{equation*}
Combining these inequalities completes the proof. 
\end{proof}

\section{Extensions and unions of hyperplanes}

In this section, we improve on some of our previous bounds in the special case that $k=n-1$. Since we are only considering hyperplanes, and to simplify some of the inequalities, we will often use the shorthand $(x, a\cdot x + b)$ for $(x_1, ..., x_{n-1}, a_1x_1+, ..., +a_{n-1} x_{n-1}+b)$. Not every hyperplane has this form, but this is no obstruction, since we can decompose a union of hyperplanes into finitely many pieces and rotate each piece so that all of its hyperplanes are of this form. 

Key to the proofs of both Theorem \ref{thm:optimalExtension} and Theorem \ref{thm:unionHyperplanes} is the following inequality:

\begin{prop}\label{prop:everyPrecisionBound}
    Let $a=(a_1, ..., a_{n-1})\in\mathbb{R}^{n-1}$, $b\in\mathbb{R}$, $x=(x_1, ..., x_{n-1})\in\mathbb{R}^{n-1}$, $\ve>0$ and an oracle $A\subseteq\mathbb{R}^n$ be given. Assume $\dim^{A, a, b}(x)\geq n-1-\frac{\ve}{4}$. Then there exists some constant $C$ such that for every precision $r$,
    \begin{equation*}
        K_r^A(x, a\cdot x+b)\geq (n-1) r  + K^A_{c_r}(a, b)+(r - c_r) -C \sqrt{\ve}r,
    \end{equation*}
    where $c_r$ denotes the largest minimizer of $K^A_t(a, b) - t$ for $t\in[1, r]$. 
\end{prop}
\noindent Note that, when $n=2$, this is just inequality 4.6 in \cite{bushling2025extension}. 

\subsection{Proofs of the hyperplane theorems assuming Proposition \ref{prop:everyPrecisionBound}:} We will briefly delay the proof of Proposition \ref{prop:everyPrecisionBound} in order to first see how Theorem \ref{thm:optimalExtension} and Theorem \ref{thm:unionHyperplanes} follow from it. 

\begin{T5}
    Let $E\subseteq\mathbb{R}^n$ be given, and let $F$ denote the union of $E$ with every $(n-1)$-plane that intersects $E$ in a set of Hausdorff dimension $n-1$. Then, 
    \begin{equation*}
        \dim_P(F)=\dim_P(E).
    \end{equation*}
\end{T5}

As indicated in the introduction, this proof is largely identical to that of the main planar result in \cite{bushling2025extension}; the novelty in this instance comes more from one of the key tools used in the proof of Proposition \ref{prop:everyPrecisionBound}. For the purpose of completeness, however, we include the proof.

\begin{proof}
The effectivization of Theorem \ref{thm:optimalExtension} is the same as that of Theorem \ref{thm:basicExtension}; in particular, it suffices to show that for every hyperplane $(a, b)$, oracle $A$, and $x_\ve\in\mathbb{R}^{n-1}$ such that $\dim^A(x_\ve)\geq n-1-\ve$,
\begin{equation*}
    \Dim^A(x_\ve, a\cdot x_\ve + b) \geq \sup_{x\in\mathbb{R}^{n-1}} \Dim^A(x, a\cdot x + b) - C \sqrt{\ve},
\end{equation*}
for some constant $C$. This inequality in turn follows if for every $x_\ve$ as above, every $x\in\mathbb{R}^{n-1}$, and every $r$ sufficiently large,
\begin{equation}\label{eq:optimalMain}
    K^A_r(x_\ve, a\cdot x_\ve + b) \geq K^A_r(x, a\cdot x + b) - C\sqrt{\ve} r,
\end{equation}
for some constant $C$. We have at our disposal two facts that let us upper bound $K^A_r(x, a\cdot x + b)$ for arbitrary $x$. Assume $t\leq r$. First, we know
\begin{equation*}
    K^A_t(x, a, b) \geq K_t^A(x, ax + b) - O(\log r).
\end{equation*}
Secondly, 
\begin{equation*}
    n(r-t) \geq K^A_{r, t}(x, ax + b) - O(\log r).
\end{equation*}
Applying these with $t=c_r$ yields
\begin{equation*}
    n(r-c_r) + K^A_{c_r}(x, a, b) \geq  K^A_{r, t}(x, ax + b) - O(\log r).
\end{equation*}
Since $x$ is an $n-1$-dimensional object, we have that
\begin{equation*}
    n(r-c_r) + (n-1) c_r + K^A_{c_r}(a, b) \geq  K^A_{r, t}(x, ax + b) - O(\log r).
\end{equation*}
Comparing this inequality with the bound on $K_r^A(x_\ve, a \cdot x_\ve + b)$ from Proposition \ref{prop:everyPrecisionBound} establishes \eqref{eq:optimalMain} and completes the proof. 
\end{proof}

Now, we move on to our final theorem's proof. 

\begin{T6}
Suppose $\mathcal{P}\subseteq\mathcal{A}(n, n-1)$ has packing dimension $t>0$. Let $F=\bigcup_{P\in\mathcal{P}}P^\prime$, where each $P^\prime\subseteq P$ has Hausdorff dimension $n-1$. Then,
    \begin{equation*}
        \dim_P(F)\geq n-1 + \frac{nt}{(n-1)t +n}.
    \end{equation*}
\end{T6}

As compared to the proof of Theorem \ref{thm:optimalExtension} assuming Proposition \ref{prop:everyPrecisionBound}, which largely mirrored the proof of its two-dimensional analog in \cite{bushling2025extension}, the proof of Theorem \ref{thm:unionHyperplanes} requires some newer ideas. As mentioned in the introduction, we have to be rather careful with the exact precisions at which we apply Proposition \ref{prop:everyPrecisionBound}. Furthermore, we will need to have some control over what $c_r$ actually is in these applications, which will require that we introduce an extra oracle to lower the complexity of $(a, b)$ at certain precisions. Lemma 8 of \cite{lutz2020bounding} (which we state in the form used in \cite{stull2022pinned}) is ideal for this purpose. 

\begin{lem}\label{lem:oracleD}

    Let $A \subseteq \{0,1\}^*$, $r \in \N$, $z \in \R^n$, and $\eta \in \Q^+$. There is an oracle $D = D(A, n, r, z, \eta)$ satisfying the following:
    \begin{enumerate}[label={\normalfont \textbf{(\arabic*)}}, itemsep=1.5pt, topsep=-3pt]
        \item For every natural number $t \leq r$, 
        \begin{equation*}
            K_t^{A,D}(z) = \min \, \{ \eta r, K_t^A(z) \} + O(\log r).
        \end{equation*}
        \item For every $m,t \in \N$ and $y \in \R^m$, 
        \begin{equation*}
            K_{t,r}^{A,D}(y \mid z) = K_{t,r}^A(y \mid z) + O(\log r) \quad \text{and} \quad K_t^{A,D,z}(y) = K_t^{A,z}(y) + O(\log r).
        \end{equation*}
        \item If $B \subseteq \{0,1\}^*$ satisfies $K_r^{A,B}(z) \geq K_r^A(z) - O(\log r)$, then
        \begin{equation*}
            K_r^{A,B,D}(z) \geq K_r^{A,D}(z) - O(\log r).
        \end{equation*}
        \item For every $m,t \in \N$, $u \in \R^n$, and $w \in \R^m$,
        \begin{equation*}
            K_{r,t}^A(u \mid w)\leq K_r^{A,D}(u \mid w) + K_{r,t}^A(z) - \eta r + O(\log r).
        \end{equation*}
    \end{enumerate}
\end{lem}

Equipped with this tool, we will find precisions $r_i$ where $K_{r_i}^A(a, b)$ is large and use an oracle $D$ to flatten the complexity function for precisions larger than $r_i$.  Then, we will apply Proposition \ref{prop:everyPrecisionBound} at essentially the least precision $t_i$ at which we can make use of all the complexity $K_{r_i}^A(a, b)$. This strategy ensures the quantity $\frac{K_{t_i}^A(x, a\cdot x + b)}{t_i}$ is large, which will allow us to establish a better bound on the effective packing dimension of $(x, a\cdot x + b)$.

\begin{proofof}{Theorem \ref{thm:unionHyperplanes}}
By the point-to-set principle and Proposition \ref{prop:hyperplaneEquiv}, it suffices to show that for every $(a, b)$, oracle $A$, and $S\subseteq\mathbb{R}^{n-1}$ of Hausdorff dimension $n-1$,
\begin{equation}\label{eq:hyperplaneMainBound}
     \sup_{x\in S} \Dim^A(x, a\cdot x + b) \geq \frac{\Dim^A(a, b) n}{(n-1)\Dim^A(a, b) + n} + (n-1).
\end{equation}
Since $S$ has Hausdorff dimension $(n-1)$, given $\ve>0$, we may assume throughout the proof that $x$ satisfies $\dim^{A, a, b}(x)\geq n-1-\frac{\ve}{4}$. Let $r_i$ be a sequence of precisions such that
\begin{equation}\label{eq:realizesPackingDimension}
\left(\Dim^A(a, b)-\frac{1}{i}\right) r_i \leq K^A_{r_i}(a, b) \leq \left(\Dim^A(a, b)+\frac{1}{i}\right) r_i.
\end{equation}
Define $c_i:=c_{r_i}$. We will consider two cases. First, suppose $c_i \geq (1 - \ve)r_i$. Then, applying Proposition \ref{prop:everyPrecisionBound}, we have
\begin{equation*}
    K^{A}_{r_i}(x, a\cdot x + b)\geq (n-1) r_i + K^{A}_{r_i}(a, b) - O(\sqrt{\ve} r_i)
\end{equation*}
which, with the choice of $r_i$ implies
\begin{equation}\label{eq:hyperplaneCase1}
    K^{A}_{r_i}(x, a\cdot x + b)\geq \left(n - 1 + \Dim^A(a, b) -\frac{1}{i} \right)r_i - O(\sqrt{\ve} r_i).
\end{equation}
We remark that, throughout the proof, the implicit constants do not depend on $i$. 

 Now we consider the case that $c_i<(1 - \ve) r_i$. Morally speaking, the worst scenario is that $K^A_{c_i}(a, b)=0$, but immediately after, the complexity function grows at the maximum possible rate until precision $r_i$. Proposition \ref{prop:everyPrecisionBound} only allows us to extract complexity from $(a, b)$ at a rate of at most $1$, which means applying it at the precision $r_i$ would be very inefficient. Instead, we want to apply this proposition at a larger precision, $t_i$, which will allow us to pick up more of the growth. We will define $t_i$ and several other precisions to reduce to this worst case and establish a bound in this instance.  

 First, let $d_i$ be such that $n(r_i - d_i) = K^A_{r_i}(a, b)$. Since $n>1$, using Proposition \ref{prop:caseLutz}, it is easy to see that 
\begin{equation}\label{eq:alternatePrecision}
c_i - K^A_{c_i}(a, b)\leq d_i + O(\log r_i).
\end{equation}
Second, let $t_i$ be such that $t_i - d_i = K^A_{r_i}(a, b)$. We want to use an oracle $D$ to lower the complexity $K^A_{t_i}(a, b)$ to roughly $K^A_{r_i}(a, b)$. More precisely, let $\eta_i$ be chosen (from some finite set of rationals $Q_\ve$ depending only on $\ve$) such that  
\begin{equation*}
    \dfrac{K^A_{r_i}(a, b)}{t_i} - 2\ve <\eta_i <  \dfrac{K^A_{r_i}(a, b)}{t_i} - \ve.
\end{equation*}
Then define the oracle $D_i:=D(A, n, t_i, (a, b),\eta_i)$  as in Lemma \ref{lem:oracleD}. It is easy to check that our definition of $D_i$ ensures that $c_{t_i}$ (relative to $(A, D_i)$) is approximately $t_i$. To do so, note that if $t^\prime_i$ is such that $K^A_{c_i}(a, b) + t^\prime_i - c_i= K^A_{r_i}(a, b)$, then \eqref{eq:alternatePrecision} guarantees $t_i^\prime\leq t_i + O(\log r)$, so $K^{A, D_i}_s(a, b)\geq K^A_{t_i}(a, b) - t_i + s$ for all $s\in[1, (1 - \ve)r]$. Hence, for sufficiently large $i$, by Proposition \ref{prop:everyPrecisionBound},
\begin{align}
    K^{A}_{t_i}(x, a\cdot x + b) &\geq K^{A, D_i}_{t_i}(x, a\cdot x + b) + O(\log r_i)\nonumber \\
    &\geq (n-1) t_i + K^{A, D_i}_{c_{t_i}}(a, b) + t_i - c_{t_i} - O(\sqrt{\ve} r_i)\nonumber \\
    &\geq (n-1) t_i + K^{A, D_i}_{t_i}(a, b) - O(\sqrt{\ve} r_i)\nonumber \\
    &= (n-1) t_i + K^A_{r_i}(a, b) - O(\sqrt{\ve} r_i) \nonumber \\
    &\geq (n-1) t_i + (\Dim^A(a, b) - \frac{1}{i})r_i - O(\sqrt{\ve} r_i).\label{eq:boundWithTi}
\end{align}
Thus, our remaining goal is to bound $r_i$ in terms of $t_i$. Using \eqref{eq:realizesPackingDimension}, the definitions of $d_i$ and $t_i$, and some basic algebra, we have
\begin{equation*}
    r_i \geq \dfrac{n}{(n-1) \Dim^A(a, b) + n} t_i - O(\frac{r_i}{i}).
\end{equation*}
Combining this with \eqref{eq:boundWithTi} and noting that $r_i$ and $t_i$ are comparable gives
\begin{equation}\label{eq:hyperplaneCase2}
    K^{A}_{t_i}(x, a\cdot x + b)\geq (n-1) t_i +  \dfrac{n \Dim^A(a, b)}{(n-1) \Dim^A(a, b) + n} t_i- O\left(\sqrt{\ve} t_i + \frac{t_i}{i}\right).
\end{equation}

It is easy to see that bound \eqref{eq:hyperplaneCase1} in the first case is always better than the bound in \eqref{eq:hyperplaneCase2} in the second case, so we can conclude that there is an infinite sequence of precisions $s_i$ at which
\begin{equation*}
    K^{A}_{s_i}(x, a\cdot x + b)\geq (n-1) s_i +  \dfrac{n \Dim^A(a, b)}{(n-1) \Dim^A(a, b) + n} s_i- O\left(\sqrt{\ve} s_i + \frac{s_i}{i}\right),
\end{equation*}
which implies
\begin{equation*}
    \Dim^A(x, a\cdot x+  b)\geq n-1 + \dfrac{n \Dim^A(a, b)}{(n-1) \Dim^A(a, b) + n} - O(\sqrt{\ve}).
\end{equation*}
Since we can choose a suitable $x\in S$ for arbitrary $\ve>0$, this establishes  \eqref{eq:hyperplaneMainBound}. \qed

\end{proofof}
\begin{remark}
    Proposition \ref{prop:everyPrecisionBound} is actually enough to establish the sharp Hausdorff dimension version of Theorem \ref{thm:unionHyperplanes} (a special case of the main bound in \cite{heraKeleti2019hausdorff}) without having to define other precisions as we did above. The desired bound in this instance is 
    \begin{equation*}
        \dim_H(F)\geq n-1 + \min\{1, t\}. 
    \end{equation*}
    In Proposition \ref{prop:everyPrecisionBound}, we have the terms 
    \begin{itemize}
    \item $(n-1)r$, 
    \item $K^A_{c_r}(a, b)$, which we can essentially lower bound by $t c_r$, and 
    \item $r - c_r$, which we can obviously lower bound by $\min\{1, t\} (r - c_r)$
    \end{itemize}
    Hence, at every precision 
    \begin{equation*}
    K^A_r(x, a \cdot x + b)\geq (n - 1 )r + \min\{t, 1\} r
    \end{equation*}
    We are unaware of different sharp examples in the packing dimension version of this problem, so we are unsure whether the difference in the Hausdorff and packing bounds we can obtain illustrates a limitation of our strategy or a distinction between the two cases. However, this does seem to be an interesting question. 
\end{remark}

\subsection{Proof of Proposition \ref{prop:everyPrecisionBound}:} Before we prove Proposition \ref{prop:everyPrecisionBound}, we will need two lemmas in addition to the aforementioned Lemma \ref{lem:oracleD}. To start, we will need an enumeration lemma, which gives sufficient conditions under which we can deduce a strong lower bound on the complexity of points lying on a hyperplane.  

\begin{lem}\label{lem:enumeration}
    Suppose $a, x\in\mathbb{R}^{n-1}$, $b \in \R$, $B \subseteq \{0,1\}^*$, $r \in \N$ is sufficiently large depending on $a, x,$ and $b$, $\delta \in \R^+$, and $\ve,\eta \in \Q^+$ satisfy the following:
    \begin{enumerate}
        \item $K_r^B(a,b) \leq (\eta + \ve) r$.
        \item For every $(u,v) \in B((a,b),1)$ such that $u\cdot x+v = a\cdot x+b$,
        \begin{equation*}
            K_r^B(u,v) \geq (\eta - \ve)r + \delta (r-t)
        \end{equation*}
        whenever $t := -\log |(a,b) - (u,v)| \in (1,r]$.
    \end{enumerate}
    Then 
    \begin{equation*}
        K_r^B(x, a\cdot x+b) \geq K_r^B(x,a,b) - \frac{2n\ve}{\delta} r - K^B(\ve) - K^B(\eta) - O(\log r).
    \end{equation*}
    \end{lem}
    \noindent This is just Lemma 6 in \cite{lutz2020bounding} modified for hyperplanes. Since the proof of this lemma is essentially the same as the proof  in \cite{lutz2020bounding}; we omit it.
    
     Finally, we will need a ``geometric'' lemma that bounds the complexity of other hyperplanes that intersect a given hyperplane $(a, b)$ at a particular point on it.  

\begin{lem}\label{lem:intersectionGeo}
    Let $n\geq 2$ and a hyperplane $(a, b)=(a_1, ..., a_{n-1}, b)$ be given and assume $(u, v)=(u_1, ..., u_{n-1}, v)$ is a hyperplane such that $a\cdot x + b= u\cdot x + v$. Let $t\leq r$ be the largest precision up to which $(a, b)$ and $(u, v)$ agree. Then for every oracle $A$, 
    \begin{equation*}
        K_r^A(u, v)\geq K_t^A(a, b) + \max\{K^A_{r-t, r}(x\mid a, b) - (n-2)(r-t), 0\}- O(\log r)
    \end{equation*}
    where the implicit constant can depend on $x, a$, and $b$. 
\end{lem}

This is a generalization of Lemma 7 from \cite{lutz2020bounding}, and the idea is relatively straightforward. If the hyperplanes $(a, b)$ and $(u, v)$ differ by roughly $2^{-t}$, the symmetry of information allows us to extract the first term on the right hand side of the inequality. Furthermore, the directions $a$ and $u$ also have to differ by roughly $2^{-t}$ (up to a multiplicative constant depending on the norm of $x_1, ..., x_{n-1}$). Hence, the $2^{-r}$-neighborhoods of the hyperplanes intersect in (very roughly) a prism of dimensions $2^{-(r-t)}\times 2^{-r}\times 1\times ... \times 1$. So, $2^{-r}$ approximations of both hyperplanes and at most $(n-2)(r-t)$ bits of additional information determine a precision $(r-t)$ approximation of any point on their intersection. 

It seems possible to prove Lemma \ref{lem:intersectionGeo} directly, but we elect for a somewhat more streamlined proof that heavily leverages the symmetry of information and some of the new tools in Section 3 to gradually reduce the problem to the two-dimensional case handled by N. Lutz and Stull \cite{lutz2020bounding}.

\begin{proof}
  Throughout this proof, we will consider pairs $(a, b)$ and $(u, v)$ to be one object, meaning an expression like $K^A_{r, t}(u, v\mid a, b)$ indicates the difficulty of computing the pair $(u, v)$ to precision $r$ given the pair $(a, b)$ to precision $t$. By the symmetry of information and the definition of $t$,
\begin{align*}
    K^A_r(u, v)&= K^A_t(u, v) + K^A_{r, t}(u, v\mid u, v) - O(\log r)\\
    &= K^A_t(a, b) + K^A_{r, t}(u, v\mid a, b) - O(\log r)\\
    &\geq K^A_t(a, b) + K^A_r(u, v\mid a, b) - O(\log r).
\end{align*}
The second term on the right hand side is certainly nonnegative, so it remains to show 
\begin{equation}\label{eq:geoLemmaMain1}
K^A_r(u, v\mid a, b)\geq K^A_{r-t, r}(x\mid a, b)-(n-2)(r-t) - O(\log r).
\end{equation}
Let $P=V+s$ denote the $(n-2)$-plane determined by the intersection of $(a, b)$ and $(u, v)$. Since $x$ is on $P$, applying Lemma \ref{lem:pointsOnAPlane} we see that \eqref{eq:geoLemmaMain1} follows if
\begin{equation*}
    K_r^A(u, v\mid a, b) + K_{r-t}^A(x\mid P) \geq K^A_{r-t, r}(x\mid a, b) + O(\log r).
\end{equation*}
Our main remaining task will be to prove that approximations of our hyperplanes determine approximations of their intersections. More precisely, we will show that
\begin{equation}\label{eq:hyperplanesDetermineIntersection}
    K^A_{r-t, r, r}(P\mid a, b, u, v)\leq O(\log r),
\end{equation}
since with this inequality in hand,
\begin{align*}
    K_r^A(u, v\mid a, b) + K_{r-t}^A(x\mid P)&= K_r^A(u, v\mid a, b)+  K^A_{r-t, r, r}(P\mid a, b, u, v) + K_{r-t}^A(x\mid P)- O(\log r)\\
    &\geq K^A_{r-t, r}(P\mid a, b)+ K^A_{r-t}(x\mid P)- O(\log r)\\
    &\geq K^A_{r-t, r}(x\mid a, b)-O(\log r).
\end{align*}

We will use the fact that $P=V+s$ to split up and bound $K^A_{r-t, r, r}(P\mid a, b, u, v)$. Trivially, we have that
\begin{align*}
    K^A_{r-t, r, r}(P\mid a, b, u, v)&\leq K_{r, r-t, r, r}^A(V, s\mid a, b, u, v)) + O(\log r)\\
    &\leq K_{r}^A(V\mid a, b, u, v)) +  K_{r-t, r, r}^A(s\mid a, b, u, v)) +O(\log r).
\end{align*}
Hence, \eqref{eq:hyperplanesDetermineIntersection} follows if 
\begin{equation}\label{eq:hyperplanesDetermineIntersection1}
     K_{r}^A(V\mid a, b, u, v))\leq O(\log r),
\end{equation}
and 
\begin{equation}\label{eq:hyperplanesDetermineIntersection2}
     K_{r-t, r, r}^A(s\mid a, b, u, v))\leq O(\log r).
\end{equation}

As in the proof of Proposition \ref{prop:hyperplaneEquiv}, approximations of the hyperplanes $H_1, H_2\in \mathcal{G}(n, k)$ associated to $(a, b)$ and $(u, v)$ respectively can be calculated from $(a, b)$ and $(u, v)$. Hence
\begin{align*}
     K_{r}^A(V\mid a, b, u, v)) &\leq K_{r, r+1, r+1}^A(V\mid a, b, u, v))\\
     &\leq K_{r, r+1, r+1}^A(V\mid H_1, H_2)+O(\log r); 
\end{align*}
in the above, we have also employed Proposition \ref{prop:caseLutz} to slightly improve the precision. Projection onto $V$ is the composition of projection onto $H_1$ and projection onto $H_2$. Let $\bar{H}_1$ and $\bar{H}_2$ denote $2^{-(r+1)}$-approximations of $H_1$ and $H_2$ and let $\bar{V}$ denote their composition. It follows immediately from the definition of the metric on the Grassmannian that 
\begin{equation*}
    \rho(V, \bar{V}) \leq \rho(H_1, \bar{H}_2) + \rho(H_1, \bar{H}_2) \leq 2^{-r}.
\end{equation*}
Since there is a Turing machine that can implement this calculation, this implies
\begin{equation*}
 K_{r, r+1, r+1}^A(V\mid H_1, H_2)\leq O(\log r)
 \end{equation*}
 and establishes \eqref{eq:hyperplanesDetermineIntersection1}.

To show \eqref{eq:hyperplanesDetermineIntersection2}, we begin similarly. Again, a precision $r$ approximation of $V$ is computable from $(a, b)$ and $(u, v)$, and applying Corollary \ref{cor:planeDeterminesComplement} to $V$ implies that
\begin{equation*}
    K_{r-t, r, r}^A(s\mid a, b, u, v))\leq K_{r-t, r, r, r}^A(s\mid a, b, u, v, V^\perp)) + O(\log r).
\end{equation*}
The projections of $(a, b)$ and $(u, v)$ onto $V^\perp$ are lines, call them $\ell_1$ and $\ell_2$ respectively. There are a number of ways to determine precision $r$ approximations $\bar{\ell}_1$ and $\bar{\ell}_2$ from precision $r$ approximations of $(a, b)$ and $(u, v)$, $\bar{V}^\perp$, and a small amount of additional information, say by sampling points on the hyperplanes and then using the matrix $\bar{V}^\perp$ to project a suitable pair for each hyperplane. Hence, 
\begin{equation*}
    K_{r-t, r, r}^A(s\mid a, b, u, v))\leq K_{r-t, r, r, r}^A(s\mid \ell_1, \ell_2, V^\perp)) + O(\log r).
\end{equation*}
Let $e$ be a standard coordinate $2$-planes such that that the projection map from $V^\perp$ to $e$ has maximal Lipschitz constant. Since $\pi_e$ is computable, and a precision $r$ approximation of $V^\perp$ is enough to recover a precision $r$ approximation of $s$ from $\pi_e s$, we have that
\begin{align*}
    K_{r-t, r, r, r}^A(s\mid \ell_1, \ell_2, V^\perp))&\leq K_{r-t, r, r, r}^A(s\mid \pi_e\ell_1, \pi_e\ell_2, V^\perp)) + O(\log r)\\
    &\leq K_{r-t, r, r}^A(\pi_es\mid \pi_e\ell_1, \pi_e\ell_2)) + O(\log r).
\end{align*}
Note that, at this point, we are working entirely in $\mathbb{R}^2$. The main content of Lemma 7 in \cite{lutz2020bounding} is that precision $r$ approximations of two lines in $\mathbb{R}^2$ are sufficient to compute a precision $r-t^\prime$ approximation of their intersection, where $2^{-t^\prime}$ is the approximate distance between the lines. So, it remains only for us to show that $\rho(\pi_e\ell_1, \pi_e\ell_2)\geq C 2^{-t}$ for some $C>0$ depending only on $(a, b)$ and the norm of $x$. 

Recall that $\vert (a, b) - (u, v)\vert\geq 2^{-t}$. There is some $C_1>0$ such that $\rho(\ell_1, ell_2)\geq C_1 2^{-t}$, since $V^\perp$ being orthogonal to $(a, b)$ and $(u, v)$ implies the angle between $\ell_1$ and $\ell_2$ is the same as the angle between $(a, b)$ and $(u, v)$. Furthermore, our choice of $e$ guarantees that there exists some $C_2$ such that $\rho(\pi_e\ell_1, \pi_e\ell_2)\geq C_2 \rho(\ell_1, \ell_2)$. Hence, taking $C=C_1 C_2$ and using Lemma 7 of \cite{lutz2020bounding} implies
\begin{align*}
    K_{r-t, r, r}^A(\pi_es\mid \pi_e\ell_1, \pi_e\ell_2))&\leq K_{r-t+C, r, r}^A(\pi_es\mid \pi_e\ell_1, \pi_e\ell_2))+O(\log r)\\
    &\leq O(\log r),
\end{align*}
which establishes \eqref{eq:hyperplanesDetermineIntersection2} and completes the proof.

\end{proof}

It remains only to use these lemmas to prove Proposition \ref{prop:everyPrecisionBound}, since we have already seen that this inequality implies Theorem \ref{thm:optimalExtension} and Theorem \ref{thm:unionHyperplanes}. Modulo the new geometric lemma, this proof is almost identical to the ``choosing an oracle'' and ``lower bound'' portions of \cite{bushling2025extension}; we include it (in a condensed form) for the purpose of completeness.

\begin{proofof}{Proposition \ref{prop:everyPrecisionBound}}

    Let $\ve>0$ and $r$ sufficiently large be given. Recalling the definition of $c_r$, we will choose $0\leq \eta<1$ such that $\eta =\frac{i}{2^m}$ for $m:=2 - \lceil \log(\ve) \rceil$ and some integer $i$. Letting $d_r:=\frac{r - c_r + K^A(a, b)_{c_r}}{r}$, we pick $\eta$ such that $d_r- 2\sqrt{\ve}<\eta<d_r- \sqrt{\ve}$ if it is possible to do so respecting the previous requirement, and $\eta = 0$ otherwise. Choose the oracle $D:=D(A, n, r, (a, b), \eta)$ as in Lemma \ref{lem:oracleD}. Working relative to this oracle, we have the following two inequalities, which are (4.4) and (4.5) from \cite{bushling2025extension}:
    \begin{equation}\label{eq:boundAtR}
        K^{A, D}_r(a, b)\geq K^A_{c_r}(a, b) + (r - c_r) - 2\sqrt{\ve} r  - O(\log r),
    \end{equation}
    and for $s\in[1, r]$,
    \begin{equation}\label{eq:betterThanTeal}
        K^{A, D}_s(a, b)\geq K^{A, D}_r(a, b) - (1 - \sqrt{\ve})(r - s) - O(\log r).
    \end{equation}
The second inequality, in particular, will allow us to apply Lemma \ref{lem:enumeration} relative to $(A, D)$. The first condition of this lemma is satisfied for sufficiently large $r$ by the definition of the oracle $D$. To see that the second is satisfied, let $(u, v)$ be given. Applying Lemma \ref{lem:intersectionGeo}, we have that
\begin{equation*}
    K_r^{A, D}(u, v)\geq K_t^{A, D}(a, b) + K^{A, D}_{r-t, r}(x\mid a, b) - (n-2)(r-t)- O(\log r),
\end{equation*}
where $t$ is as in the statement of the lemma. The properties of the oracle $D$ and \eqref{eq:betterThanTeal} ensure that
\begin{equation*}
    K_r^{A, D}(u, v)\geq K^A_r(a, b) + (1 - \sqrt{\ve})(r - t) + K_{r - t}^{A, a, b}(x) - (n-2)(r-t) - O(\log r).
\end{equation*}

Recall that $x$ is an $(n-1)$-dimensional object which we assumed was almost random relative to $(A, a, b)$. Provided that $r$ is sufficiently large, then, we have
\begin{equation*}
    K_r^{A, D}(u, v)\geq K^A_r(a, b) + (1 - \sqrt{\ve})(r - t) + (1 - \frac{\ve}{2})(r-t) - O(\log r).
\end{equation*}
which, using the definition of $\eta$, implies
\begin{equation*}
    K_r^{A, D}(u, v)\geq (\eta - \ve) r + \sqrt{\ve}(r - t).
\end{equation*}
This allows us to apply Lemma \ref{lem:enumeration}, yielding
\begin{equation}\label{eq:afterEnumeration}
     K^{A, D}(x, a\cdot x + b)\geq K^{A, D}_r(x, a, b) - 2 n \sqrt{\ve} r - O(\log r).
\end{equation}
Using the symmetry of information and the assumption on $x$, 
\begin{equation*}
    K^{A, D}(x, a, b)\geq (n - 1 - \sqrt{\ve}) r + K_r^{A, D}(a, b) - O(\log r).
\end{equation*}
Combining the above with \eqref{eq:boundAtR} and \eqref{eq:afterEnumeration} for sufficiently large $r$, we have
\begin{align*}
K^A(x, a\cdot x + b)&\geq K^{A, D}(x, a\cdot x + b)+O(\log r)\\
    &\geq (n - 1) r + K^A_{c_r}(a, b) + (r - c_r) - C\sqrt{\ve} r,
\end{align*}
for $C:=(2n + 3)$. \qed 
\end{proofof}

\section{Acknowledgments}
I am very grateful to Betsy Stovall, Don Stull, Patrick Lutz, and Paige Bright for feedback on an earlier draft of this paper. Additionally, I would like to thank Don Stull for helpful conversations about some of the material, particularly in Section 3.

\bibliographystyle{amsplain}
\bibliography{references}

 \end{document}